\documentclass[a4paper,10pt,
 ]{article}%

\usepackage[hidelinks]{hyperref}
\hypersetup{
  final,
  pdftitle={Genuine operadic nerve},
  pdfauthor={Bonventre, Peter},
  linktoc=page
}

\usepackage{amsmath, amsfonts, amssymb, amsthm}%

\usepackage{relsize}%

\usepackage{geometry}

\usepackage[latin1]{inputenc}%
\usepackage{MnSymbol}
\usepackage{mathrsfs}

\usepackage{enumerate}%
\usepackage[inline]{enumitem}%

\usepackage{tikz}%
\usetikzlibrary{matrix,arrows,decorations.pathmorphing,
cd,patterns,calc}
\tikzset{
  treenode/.style = {shape=rectangle, rounded corners,%
                     draw, align=center,%
                     top color=white, bottom color=blue!20},%
  root/.style     = {treenode, font=\Large, bottom color=red!30},%
  env/.style      = {treenode, font=\ttfamily\normalsize},%
  dummy/.style    = {circle,draw,inner sep=0pt,minimum size=2mm},%
  shorten <>/.style={shorten >=#1,shorten <=#1}
}%

\usetikzlibrary[decorations.pathreplacing]

\usepackage{ifdraft}
\ifdraft{
  \color[RGB]{63,63,63}
  \pagecolor[RGB]{220,220,204}
  
  \usepackage{showkeys}
} 
{
}

\usepackage{todonotes}

\newcommand\IG[2][]{\IfFileExists{#2}
  {\includegraphics[#1]{#2}}
  {\fbox{File #2 doesn't exist}%
   \message{Imagefile #2 doesn't exist^^J}}}

\numberwithin{equation}{section} 
\numberwithin{figure}{section}

\newtheorem{theorem}[equation]{Theorem}%
\newtheorem*{theorem*}{Theorem}%
\newtheorem{lemma}[equation]{Lemma}%
\newtheorem{proposition}[equation]{Proposition}%
\newtheorem{corollary}[equation]{Corollary}%
\newtheorem*{conjecture*}{Conjecture}%
\newtheorem{innercustomgeneric}{\customgenericname}
\providecommand{\customgenericname}{}
\newcommand{\newcustomtheorem}[2]{%
  \newenvironment{#1}[1]
  {%
   \renewcommand\customgenericname{#2}%
   \renewcommand\theinnercustomgeneric{##1}%
   \innercustomgeneric
  }
  {\endinnercustomgeneric}
}

\newcustomtheorem{customthm}{Theorem}
\newcustomtheorem{customcor}{Corollary}

\theoremstyle{definition} 
\newtheorem{definition}[equation]{Definition}%
\newtheorem*{definition*}{Definition}%
\newtheorem{example}[equation]{Example}%
\newtheorem{remark}[equation]{Remark}%
\newtheorem{notation}[equation]{Notation}%
\newtheorem{convention}[equation]{Convention}%
\newtheorem{warning}[equation]{Warning}

\usepackage{mathtools}
\mathtoolsset{showonlyrefs}

\newcommand{\V}{\ensuremath{\mathcal V}}
\newcommand{\set}[1]{\left\{ #1 \right\}}%
\newcommand{\varrow}[2]{\substack{ #1 \\ \downarrow \\ #2 }}

\newcommand{\longto}{\longrightarrow}%
\newcommand{\into}{\hookrightarrow}%
\newcommand{\F}{\ensuremath{\mathbb{F}}}%
\newcommand{\C}{\ensuremath{\mathcal{C}}}%
\renewcommand{\P}{\ensuremath{\mathcal{P}}}%
\renewcommand{\O}{\mathcal{O}}%
\newcommand{\ksi}{\xi}

\newcommand{\Fin}{\mathsf{F}}%
\newcommand{\Set}{\ensuremath{\mathsf{Set}}}
\newcommand{\Top}{\ensuremath{\mathsf{Top}}}
\newcommand{\sSet}{\ensuremath{\mathsf{sSet}}}%
\newcommand{\Cat}{\mathsf{Cat}}
\newcommand{\sCat}{\mathsf{sCat}}
\newcommand{\Op}{\mathsf{Op}}%
\newcommand{\sOp}{\ensuremath{\mathsf{sOp}}}%
\newcommand{\dSet}{\mathsf{dSet}}
\newcommand{\Fun}{\mathsf{Fun}}

\newcommand{\Comm}{\mathsf{Comm}}
\newcommand{\Fib}{\mathsf{Fib}}
\newcommand{\Alg}{\mathsf{Alg}}

\DeclareMathOperator{\Map}{Map}%
\DeclareMathOperator{\Nat}{Nat}

\DeclareMathOperator{\Emb}{Emb}

\newcommand{\Stab}{\ensuremath{\mathrm{Stab}}}%

\usepackage{harpoon}
\newcommand{\vect}[1]{\overrightharp{\ensuremath{#1}}}

\newcommand{\UV}{\underline{\mathcal V}}
\renewcommand{\phi}{\varphi}

\renewcommand{\F}{\mathcal F}

\renewcommand{\hat}{\widehat}

\newcommand{\overbar}[1]{\mkern 1.5mu\overline{\mkern-1.5mu#1\mkern-1.5mu}\mkern 1.5mu}
\renewcommand{\bar}{\overbar}

\title{The genuine operadic nerve}

\author{Peter Bonventre}%

\date{\today}

\begin{document}

\maketitle

\begin{abstract}
      We construct a generalization of the operadic nerve, providing a translation between the
      equivariant simplicially enriched operadic world to the parametrized $\infty$-categorical perspective.
      This naturally factors through genuine equivariant operads, a model for ``equivariant operads with norms up to homotopy''.
      We introduce the notion of an op-fibration of genuine equivariant operads, extending Grothendieck op-fibrations,
      and characterize fibrant operads as the image of genuine equivariant symmetric monoidal categories.
      Moreover, we show that under the operadic nerve, this image is sent to $G$-symmetric monoidal $G$-$\infty$-categories.
      Finally, we produce a functor comparing the notion of algebra over an operad in each of these two contexts.
\end{abstract}

\setcounter{tocdepth}{2}
\tableofcontents

\section{Introduction}

Operads have proven to be a valuable tool since they were introduced by
Boardman-Vogt \cite{BV73} and May \cite{May72}.
In stable homotopy theory, Boardman-Vogt and May introduced a class of simplicial operads\footnote{
  In general, we write ``operad'' and $\Op$ to refer to the category of \textit{colored} operads,
  which includes the classical single-colored notion as well as the more general notion (which often goes by the name of ``multicategory'').
  Additionally, we write ``simplicial operad'' and $\sOp$ to mean the category of (possibly many colored) \textit{operads enriched in simplicial sets}, 
  as opposed to the more general notion of a simplicial object in (colored) operads. Details will be given in Definition \ref{SOP_DEF}.}
called $E_\infty$-operads which 
encode homotopy coherent multiplication in spaces and spectra.
Further, Boardman-Vogt and May showed that spaces equipped with such structures represented infinite loop spaces.
Moreover, the homotopy theory of simplicial operads is designed so that any cofibrant replacement of the commutative operad is $E_\infty$, capturing the notion that $E_\infty$-algebras are the ``correct'' homotopical replacements of strict topological monoids.

However, while simplicial operads can encode these homotopical structures, they themselves remain fairly rigid objects.
To obtain further homotopical flexibility, an alternative framework has been pioneered by Lurie 
to work in the language of $\infty$-categories, an extension of category theory defined by Boardman-Vogt \cite{BV73} and refined by Joyal \cite{Joy02},
where notions are only ever well-defined up to coherent homotopy.
Lurie \cite{Lur17} constructs the theory of \textit{$\infty$-operads},
a certain class of $\infty$-category equipped with a map to the category of finite pointed sets $\Fin_{\**}$. 

While these two theories aim to model the same homotopy theory,
the equivalence between them was not known for some time, and is not direct.
Work of Cisinski-Moerdijk-Weiss \cite{CM13a,CM13b,CM11,MW09}, Chu-Haugseng-Heuts \cite{CHH18}, and Barwick \cite{Bar18}
produces a zig-zag of Quillen equivalences between simplicial operads and $\infty$-operads.

On the other hand, there is a fairly natural construction between these models.
For any simplicial operad $\O \in \sOp$, May-Thomason \cite{MT78} produce an associated simplicial category $\O^{\otimes}$, living over the category $\Fin_{\**}$ of pointed finite sets,
called the \textit{category of operators},
and moreover show that the theory of algebras over $\O$ and $\O^\otimes$ coincide.
The homotopy coherent nerve of $\O^\otimes$ is denoted $N^\otimes(\O)$, and is called the \textit{operadic nerve} by Lurie \cite[2.1.1.22]{Lur17}.
This construction has several nice properties.
First, in \cite[Prop. 2.1.1.27]{Lur17}, Lurie showed that this sends a levelwise fibrant simplicial operad (where all mapping spaces are Kan) to an $\infty$-operad
(and in fact, Lurie's definition of an $\infty$-operad is truly a generalization of these categories of operators).
Second, it is expected to be an equivalence of homotopy theories,
and has already been shown to be one when restricted to non-unital operads by \cite{HHM16}.
Third, the operadic nerve preserves symmetric monoidal categories.
That is, there are canonical faithful inclusions of
simplicial symmetric monoidal categories into simplicial operads and
symmetric monoidal $\infty$-categories into $\infty$-operads,
and the operadic nerve sends one subcategory to the other \cite[Prop. 4.1.7.10]{Lur17}.

In this paper, we generalize the narrative of the operadic nerve to the equivariant setting,
incorporating actions of a finite group $G$.
However, the appropriate source and target of the new map are not simply the categories of $G$-objects of source and target of the original operadic nerve.
Instead, sophisticated categories have been built to capture the more intricate homotopy theory of equivariant operads.
This additional complexity comes from the observation,
first by Constenoble-Waner \cite{CW91} and explored systematically by Blumberg-Hill \cite{BH15}, that
there are several possible notions of ``equivariant homotopy coherent multiplication''.
The distinctions come from whether or not they encode \textit{norm maps}; as a distinguished and archetypal example, in $G$-spectra these are $G$-equivariant maps of the form
\[
      G \cdot_H N^A X \simeq G \cdot_H \mathop{\bigwedge}_{a \in A} i^{\**}_H X \to X
\]
for $A$ some finite $H$-set with $H \leq G$, $X$ a $G$-spectrum, and $N^A X$ the indexed smash product,
with an $H$-action that both permutes the indices (via the action on $A$) and acts on each $X$.
So-called \textit{na\"ive $E_\infty$-operads}, $E_\infty$-operads with a trivial $G$-action, only encode norm maps for $A$ a trivial $H$-set, while
\textit{genuine $G$-$E_\infty$-operads} encode all such maps. 
As desired, there is a homotopy theory of simplicial $G$-operads which distinguishes these classes of operads,
constructed by the author and his collaborator Lu\'is Pereira \cite{BP_geo}, and independently by Gutierrez-White \cite{GW17}.

With respect to the operadic nerve, this has the following consequence.
The new source is the category $\sOp_G$ of \textit{simplicial genuine\footnote{
    The \textit{genuine} adjective here has two (related) etymologies.
    First, the image $i_{\**}\O \in \sOp_G$ of any genuine $G$-$E_\infty$-operad $\O \in \sOp^G$ is
    contractible at every level, while this fails when starting with a na\"ive $E_\infty$-operad,
    and $\sOp_G$ was designed precisely to see this distinction.
    Second, the adjective ``genuine'' has been used regularly to describe homotopy theories of equivariant objects which see all possible fixed point information, namely the genuine/fine homotopy theory of $G$-spaces and the (fully) genuine homotopy theory of $G$-spectra.
    Expanding on this, both named theories can be realized as presheaf categories,
    and in particular their homotopy groups inherit extra structure.
    Similarly, results in \cite{BP_geo} prove that the homotopy groups of simplicial $G$-operads are naturally genuine $G$-operads of sets,
    and we should think of genuine equivariant operads as appropriate analogues to presheaves in this algebraic setting.}
  equivariant operads}, also introduced in \cite{BP_geo} by the author and Pereira.
This is a generalization of simplicial $G$-operads, which embed fully-faithfully as part of a right Quillen equivalence,
where objects $\P \in \sOp_G$ have evaluations at all finite $H$-sets $A$.
Morally, the operations in $\P(A)$ precisely encode $A$-norm maps,
while the composition law details their interactions.

For the target, Barwick-Dotto-Glasman-Nardin-Shah \cite{BDGNS} have produced a beautiful theory of parametrized $\infty$-categories and parametrized homotopy theory.
In particular, when the base is the \textit{orbit category} $\mathsf O_G$ of finite transitive $G$-sets,
they recover a coherent description of equivariant homotopy theory.
Encoding algebraic structures here are $\mathsf O_G$-$\infty$-operads $\O^\otimes \in \Op_{\infty,G}$,
a certain class of $\mathsf O_G$-$\infty$-categories equipped with a map to
the category $\underline{\Fin}^G_{\**}$ of all finite pointed $H$-sets for all $H \leq G$ (cf. Definition \ref{FGP_DEF}).

Given a simplicial genuine equivariant operad $\P \in \sOp_G$,
we construct an analogue of the operadic nerve $N^{\otimes}(\P)$ dubbed the \textit{genuine operadic nerve},
also built as the homotopy coherent nerve of a (genuine) category of operators construction.
The main results of this paper are the following extensions of \cite[Prop. 2.1.1.27 and 4.1.7.10]{Lur17},
providing a (1-categorical) translation between these two theories of homotopical equivariant operads.

First, we prove the following in Section \ref{PROOFI_SEC}.
\begin{customthm}{I}
      \label{THMI}
      The genuine operadic nerve is a faithful functor
\[
      N^{\otimes}(-): \sOp_{G,f} \longto \Op_{\infty,G}
\]
from the category of level fibrant genuine equivariant operads to the (1)-category of $\mathsf O_G$-$\infty$-operads,
which recovers the original operadic nerve in the case where $G = \**$.
\end{customthm}

Additionally, similarly to the inclusion $\mathsf{sSymMon} \into \sOp$ of (simplicial) symmetric monoidal categories into (colored, simplicial) operads, 
there are analogous notions of ``symmetric monoidal category'' inside $\sOp_G$ and $\Op_{\infty,G}$,
namely the (simplicial) \textit{$E \Sigma_G$-algebras} of the author and Pereira \cite{BPGSym} and
\textit{$G$-symmetric monoidal $G$-$\infty$-categories} of Barwick et al. \cite{Nar17}.
We prove the following in Section \ref{PROOFII_SEC}, which says that the genuine operadic nerve preserves these notions of symmetric monoidal categories and the corresponding notions of symmetric monoidal functors.
\begin{customthm}{II}
      \label{THMII}
      The functor $N^{\otimes}: \sOp_G \to \Op_{\infty,G}$ restricts to a functor
      \[
            N^{\otimes}: \mathsf{sSymMon}^q_{G,f} \longto \mathsf{SymMon}_{\infty,G}.
      \]
\end{customthm}

A major step in the proof of Theorem \ref{THMII} is the identification of the image of $E\Sigma_G$-algebras inside $\sOp_G$
as those genuine operads which are ``fibered'' over the terminal operad,
in a sense which generalizes Grothendieck op-fibrations of categories, made precise in Section \ref{GENOPFIB_SEC}.

Essentially, we prove the following; a precise statement can be found at Theorem \ref{THMIII_PRECISE}.
\begin{customthm}{III}
      \label{THMIII}
      The image of simplicial $E \Sigma_G$-algebras in genuine equivariant operads
      are those operads such that the canonical map to the terminal operad
      is a genuine operadic fibration.
\end{customthm}

All together, using the notation to be defined in the paper, these results amalgamate into the following commuting diagram of well-defined faithful functors.
\begin{equation}
      \begin{tikzcd}
            \mathsf{sPerm}_{G,f} \arrow[d, "\simeq"'] \arrow[r, "\P_{(-)}", "\cong"']
            &
            \Fib^f(\Comm) \arrow[d, "\simeq"'] \arrow[r, "{(-)^{\otimes}}"]
            &
            \Fib^f_{\mathrm{Segal}}(\underline{\Fin}^G_{\**}) \arrow[d, hookrightarrow] \arrow[r, "N"]
            &
            \mathsf{SymMon}_{\infty,G} \arrow[d, equal]
            \\
            \mathsf{sSymMon}^q_{G,f} \arrow[r, "\P_{(-)}", "\cong"']
            &
            \Fib^q(\Comm) \arrow[d, hookrightarrow] \arrow[r, "{(-)^{\otimes}}"]
            &
            \Fib^q_{\mathrm{Segal}}(\underline{\Fin}^G_{\**}) \arrow[d, hookrightarrow] \arrow[r, "N"]
            &
            \mathsf{SymMon}_{\infty,G} \arrow[d, equal]
            \\
            &
            \Fib(\Comm) \arrow[d, hookrightarrow] \arrow[r, "{(-)^{\otimes}}"]
            &
            \Fib_{\mathrm{Segal}}(\underline{\Fin}^G_{\**}) \arrow[d, hookrightarrow] \arrow[r, "N"]
            &
            \mathsf{SymMon}_{\infty,G} \arrow[d, hookrightarrow]
            \\
            &
            \Op_G \arrow[r, "{(-)^{\otimes}}"]
            &
            \mathsf{CatOp}_G \arrow[r, "N"]
            &
            \Op_{\infty,G}                  
      \end{tikzcd}
\end{equation}

We end by showing that this framework preserves algebras over operads, if we make small additional assumptions on $\V$ and our operads $\O$.
In Section \ref{GSTRICT_SEC}, we build a model $\UV_{\infty,G}^\square$ for the $G$-symmetric monoidal $G$-$\infty$-category of strict $G$-objects in \textit{globally fibrant} symmetric monoidal simplicial categories $\V$ (Definition \ref{GLOBAL_DEF}), 
and prove the following for algebras over \textit{graph fibrant} operads $\O$ (Definition \ref{GRAPH_DEF}).

\begin{customthm}{IV}
      \label{THMIV}
      For any graph fibrant equivariant simplicial operad $\O \in \sOp^G$ and globally fibrant symmetric monoidal simplicial category $(\V,\square)$,
      there exists a functor of $\infty$-categories
      \[
            N \mathsf{Alg}_{\O}(\V^G) \longto \mathsf{Alg}_{N^\otimes(\O)}(\UV_{\infty,G}^{\, \square})
      \]
      between associated categories of algebras.
\end{customthm}

\begin{remark}
      We note that these results are not yet homotopical.
      However, as in the non-equivariant case, we expect that $N^{\otimes}$ is an equivalence of homotopy theories.
\end{remark}

\begin{remark}
      The author's joint work with Lu\'is Pereira provides another model of equivariant higher algebra,
      generalizing the dendroidal sets perspective of Moerdijk, Weiss, Cisinski, and Heuts \cite{CHH18,CM13a,CM13b,CM11,Heu,MW09}
      which has seen much success.
      A homotopical analogue of Theorem \ref{THMI} in this context is the precisely the culmination of the papers \cite{BP_sq,BP_edss,Per18},
      that the homotopy coherent dendroidal nerve between equivariant simplicial operads and equivariant dendroidal sets is a right Quillen equivalence.
\end{remark}

\begin{remark}
      These structures --- $E \Sigma_G$-algebras, genuine equivariant operads, $G$-$\infty$-operads, and $G$-symmetric monoidal $G$-$\infty$-categories --- have corresponding notions for any \textit{(weak) indexing system} $\mathcal F$
      in the sense of \cite{BH15}, \cite{BH18}, \cite{Per18}, \cite{BP_geo},
      i.e. replacing all instances of the category $\underline{\Fin}^G_{\**}$ of all finite pointed $G$-sets with the category
      $\underline{\Fin}^\F_{\**}$ of those finite pointed $G$-sets generated by $\mathcal F$.
      We expect the results to extend to these settings.
\end{remark}

\subsection{Organization}

The paper is planned as follows.

We begin by recalling the relevant parts of the non-equivariant story in Section \ref{NONEQ_SEC}.
In Section \ref{EPRELIM_SEC}, we discuss equivariant generalizations of the key players from Section \ref{NONEQ_SEC},
namely the category of finite pointed $G$-sets,
the category of (colored) simplicial genuine equivariant operads $\P \in \sOp_G$ from \cite{BP_geo},
and the $\mathsf O_G$-$\infty$-operads from \cite{BDGNS,Nar17}.

In Section \ref{GON_SEC},
we introduce our main constructions, 
the \textit{genuine equivariant category of operators} $\P^{\otimes}$
and the \textit{genuine operadic nerve} $N^\otimes(\P)$ associated to $\P$,
and prove Theorem \ref{THMI}.
In Section \ref{GENOPFIB_SEC} we define and analyze fibrations in $\sOp_G$.
Section \ref{GOF_SEC} defines genuine operadic op-fibrations,
Section \ref{GSYM_SEC} recalls $E \Sigma_G$-algebras from \cite{BPGSym}
with Proposition \ref{PERMG_OPG_PROP} giving the faithful inclusion of simplicial $E\Sigma_G$-algebras into $\sOp_G$,
and Section \ref{PROOFIII_SEC} discusses how the different varieties of $E \Sigma_G$-algebras can be identified with specific classes of
fibrations in $\Op_G$, culminating in the proof of Theorem \ref{THMIII}.
Theorem \ref{THMII} is finally proved in Section \ref{PROOFII_SEC}.

Lastly, in Section \ref{EXAMPLE_SEC}, we give several examples of $G$-$\infty$-operads and $G$-symmetric monoidal $G$-$\infty$-categories coming from the genuine operadic nerve,
introduce categories of algebras, and prove Theorem \ref{THMIV}.

\subsection{Acknowledgments}

This paper owes much to the ongoing collaborations with Lu\'is Pereira; the author would like to thank him for his help and useful discussions throughout their work together.
In particular, the definitions of genuine equivariant operads (single- and many-colored) and $E\Sigma_G$-algebras are joint with him.

The author would like to thank Asaf Horev for asking whether there was a conversion between $G$-operads and $G$-$\infty$-operads, leading to this work, for helpful conversations, and for sharing notes of work in progress with Inbar Klang and Foling Zou, which in particular influenced the exposition leading up to Definition \ref{VINFTYG_DEF}.
The author would also like to thank Kate Ponto, Bert Guillou, and Nat Stapleton for their comments and suggestions.

\section{Motivation and background}
\label{MOTV_SEC}
\label{NONEQ_SEC}

We begin by recalling the story in the non-equivariant setting (e.g. \cite{Lur17}), which will
provide the guiding outline of the necessary results (as well as the style of proof)
for many parts of this article.

\subsection{Categorical fibrations}

First, we recall and establish our terminology for the various different notions of ``fibrations'' of categories,
as generalizations of these ideas appear throughout this paper in many forms.%

\begin{definition}
      \label{CATFIB_DEF}
      Given a functor $p: \mathcal C \to \mathcal B$, an arrow $f: c \to c'$ in $\mathcal C$ is call \textit{$p$-cocartesian} if
      for any $b \in \mathcal B$, preimage $c'' \in p^{-1}(b)$, and arrows $g: p(c') \to b$ in $\mathcal B$ and $h: c \to c''$ in $\mathcal C$ such that $p(h) = g p(f)$,
      there exists a unique lift $\bar g: c' \to c''$ such that $p(\bar g) = g$ and $\bar g f = h$.
      \begin{equation}
            \begin{tikzcd}
                  p(c) \arrow[rr, "{p(f)}"] \arrow[dr, "{p(h)}"']
                  &&
                  p(c') \arrow[dl, "\forall g"]
                  && 
                  c \arrow[rr, "f"] \arrow[dr, "\forall h"']
                  &&
                  c' \arrow[dl, dashed, "\exists! \bar g"]
                  \\
                  &
                  b
                  &
                  && 
                  &
                  c''
            \end{tikzcd}
      \end{equation}

      We say $p$ is a \textit{Grothendieck op-fibration} if
      for every arrow in $\mathcal B$ and lift of the domain to $\mathcal C$, there exists some $p$-cocartesian lift.
      We say $p$ is additionally \textit{$q$-split} if we have a chosen system of cocartesian lifts which are natural in the arrows of $\mathcal B$.
      Finally, $p$ is additionally \textit{fully split} if this chosen system is closed under composition.

      Given two $q$-split (resp. fully split) op-fibrations $p$ and $p'$, a functor $F: \mathcal C \to \mathcal C'$ is called a
      \textit{map of (split) op-fibrations}
      if $F$ is a functor over $\mathcal B$ and sends (chosen) cocartesian arrows to (chosen) cocartesian arrows.
      We denote the 1-categories of simplicial fully-split and $q$-split op-fibrations over $\mathcal B$ with maps of $op$-fibrations by
      $\Fib^f(\mathcal B) \subseteq \Fib^q(\mathcal B) \subseteq \sCat \downarrow \mathcal B$.

      Dually, an arrow $f$ is \textit{$p$-cartesian} if $f$ is $p$-cocartesian in $\mathcal C^{op}$,
      and $p$ is a \textit{Grothendieck fibration} if $p^{op}$ is a Grothendieck op-fibration.
      Explicitly, $f: c'' \to c'$ is $p$-cartesian if for every $g: b \to p(c')$ and $h: c'' \to c'$ with $p(c'') = b$ and $p(h) = p(f) g$,
      there exists a unique lift $\bar g: c'' \to c'$ of $g$ such that $f \bar g = h$.
\end{definition}

The Grothendieck construction provides an equivalent characterization.
\begin{theorem}
      \label{OPFIB_FUN_THM}
      The category of functors $\Fun(\mathcal B^{op}, \Cat)$ is isomorphic to
      the categories of fully-split Grothendieck fibrations over $\mathcal B$ (dually, fully-split Grothendieck op-fibrations over $\mathcal B^{op}$)
      with maps of split (op)-fibrations.

      More generally, the 2-category of pseudofunctors $\mathcal B^{op} \to \Cat$ is strictly 2-equivalent to the 2-category of
      fibrations over $\mathcal B$ (dually, op-fibrations over $\mathcal B^{op}$)
      with maps of (op)-fibrations and natural transformations.
\end{theorem}

More details on these notions can be found in e.g. \cite{Gra66}.

We can repackage Definition \ref{CATFIB_DEF} as follows:
\begin{lemma}
      Given a functor $p: \mathcal C \to \mathcal B$, An arrow $f: c \to c'$ is $p$-cocartesian if the diagram
      \begin{equation}
            \label{CAT_COCART_EQ}
            \begin{tikzcd}
                  \mathcal C(c', c'') \arrow[r, "f^{\**}"] \arrow[d, "F"']
                  &
                  \mathcal C(c, c'') \arrow[d, "F"]
                  \\
                  \mathcal B(Fc', Fc'') \arrow[r, "{Ff^{\**}}"]
                  &
                  \mathcal B(Fc, Fc'')
            \end{tikzcd}
      \end{equation}
      is a strict pullback.
\end{lemma}

\begin{remark}
      \label{SCAT_COCART_REM}
      Beardsley-Wong \cite{BW} show that we can extend these definitions and discussions to the category $\mathsf{sCat}$ of categories enriched in simplicial sets (or actually any bicomplete closed symmetric monoidal category $\V$):
      Given a functor $p: \mathcal C \to \mathcal D$ of simplicially enriched categories, an arrow $f \in \mathcal C(c, c')_0$ is \textit{$p$-cocartesian} if \eqref{CAT_COCART_EQ} is a pullback in $\sSet$.
      
      We define \textit{$p$-cartesian} arrows, \textit{(q-split, fully split) Grothendeick (op)-fibrations}, and \textit{maps} of (q-split, fully split) (op)-fibrations exactly as in Definition \ref{CATFIB_DEF}.
      
      A main result \cite[Thm. 5.9]{BW} is precisely the analogue of Theorem \ref{OPFIB_FUN_THM} to this context,
      so in particular we can freely move between fully split fibrations over $\mathcal B$ and functors $\mathcal B^{op} \to \mathsf{sCat}$.
\end{remark}



\subsection{Colored simplicial operads}
\label{CSO_SEC}

We begin with our original object of study, a colored simplicial operad.
\begin{definition}
      \label{SOP_DEF}
      Given any set $\mathfrak C$, a \textit{$\mathfrak C$-signature} is a sequence $(x_1,\dots, x_n; x)$ of length $n+1$ of elements in $\mathfrak C$;
      we call the first $n$ objects the \textit{source} of the signature, and the last one the \textit{target}.
      
      A \textit{colored simplicial operad} \footnote{
        These have also been called \textit{multicategories enriched in simplicial sets}.}
      $\O \in \sOp$ consists of
      \begin{itemize}
      \item a set $\mathfrak C = \mathfrak C_\O$ of colors (or objects);
      \item for each $\mathfrak C$-signature $\underline{C} = (x_1,\dots,x_n; x)$ of colors of length $n+1$, a simplicial set $\O(\underline{C}) \in \sSet$ of operations of arity $n$;            
      \item for all collections of $\mathfrak C$-signatures of the form $\vect{C} = (x_1,\dots, x_n; x)$, $\vect{D}_i = (x^i_1,\dots, x^i_{m_i}; x_i)$ for $1 \leq i \leq n$,
            a composition law
            \[
                  \O(\vect{C}) \times \prod_i \O(\vect{D}_i) \to \O(\vect{C} \circ (\vect{D}_i))
            \]
            where
            \[
                  \vect{C} \circ (\vect{D}_i) = (x^1_1, \dots, x^1_{m_1}, x^2_1, \dots, \dots, x^n_{m_n}; x);
            \]
      \item a unit operation $1_x \in \O(x;x)$ for all colors $x \in \mathfrak C$; and
      \item for each $\sigma \in \Sigma_n$ and sequences $C$ of length $n$, an action map
            \[
                  \O(\vect{C}) \xrightarrow{\ \sigma \ } \O(\sigma^{\**}\vect{C}) = \O(x_{\sigma^{-1}1}, \dots, x_{\sigma^{-1}n}; x);
            \]
      \end{itemize}
      such that the actions of $\Sigma_n$ are unital and associative,
      composition is unital and associative,
      and composition commutes with the action of $\Sigma_n$.

      A map of operads $F: \O \to \P$ is given by
      a map of sets $F_0: \mathfrak C_\O \to \mathfrak C_\P$, and
      maps $F(\vect{C}): \O(\vect{C}) \to \P(F_0^{\**}\vect{C})$ for all $\mathfrak C_\O$-signatures $\vect{C}$
      which are compatible with all of the above structure.
\end{definition}

See Section \ref{EXAMPLE_SEC} for some examples of (equivariant) simplicial operads.

We note that there is a natural path-component functor $\pi_0$,
and a forgetful functor $j^{\**}$ which only remembers the operations of arity exactly 1.
\[
      \pi_0: \sOp \longto \Op,
      \qquad
      j^{\**}: \sOp \to \sCat.
\]

Cisinski-Moerdijk have shown that $\sOp$ has a model structure given by the following:
\begin{definition}[{\cite{CM13b}}]
      A map $F: \O \to \P$ is called a
      \begin{itemize}
      \item \textit{weak equivalence} if
            $F(\vect{C})$ is a weak equivalence in $\sSet$ for all $\mathfrak C$-signatures $\vect{C}$, and $\pi_0 j^{\**} F$ is an equivalence of categories.
      \item \textit{fibration} if $F(\vect{C})$ is a fibration in $\sSet$ for all $\mathfrak C$-signatures $\vect{C}$,
            and $\pi_0 j^{\**} F$ is an \textit{isofibration} of categories; i.e. $F$ can lift isomorphisms.
      \end{itemize}

      We say $\O \in \sOp$ is \textit{fibrant} if the map $\O \to \**$ is a fibration; i.e. $\O$ is locally fibrant, in that every hom-space is a Kan complex.
\end{definition}

\subsection{Infinity operads}
\label{INFTYOP_SEC}

The original operadic nerve construction provided a translation between this world of homotopical algebra
with the theory of $\infty$-categories.
We introduce this second setting now;
a more thorough discussion can be found at the original source \cite[\S 2]{Lur17}.

\begin{definition}
      \label{FIN_DEF}
      \label{COCART_DEF}
      We outline some basic concepts in $\infty$-category theory we will need:
      pointed finite sets, cocartesian arrows, and finally $\infty$-operads.

      \begin{itemize}
      \item Let $\mathsf F$ denote a fixed category of finite ordered sets and unordered set maps,
            such that the subcategory with ordered maps is \textit{skeletal};
            i.e., if ever two sets in $\mathsf F$ are order isomorphic, they are in fact equal.
            In particular, we may choose $\mathsf F$ to be the category with objects $\underline n = \set{1,2,\dots,n}$ for all $n \geq 1$ with unordered maps.
            
            Let $\mathsf F_s$ and $\Sigma$ denote the subcategories of $\mathsf F$ consisting of only surjective maps and bijections, respectively.

            These models are chosen so that all of the above have canonical choices for all small limits and colimits using lexicographical ordering.
            In particular, $\Fin$ is \textit{bipermutative} with respect to cartesian product and disjoint union.
            
      \item Let $\Fin_{\**}$ denote the category of pointed finite sets $A_+ = \set{0} \amalg A$ with $A \in \mathsf F$, and pointed maps.
      \item A map $f: A_+ \to B_+$ in $\Fin_{\**}$ is called \textit{inert} if $f$ is surjective and $f$ is injective away from the basepoint,
            i.e. for all $\** \neq b \in B$, $|f^{-1}(b)| \leq 1$.
      \item A map $f: A_+ \to B_+$ is called a \textit{projection map} if $B = \**$ and
            $f(a)$ is not the basepoint of $B_+$ for exactly one $r \in A$;
            in this case, we denote $f$ by $\pi_r$.

      \item Given a map of simplicial sets $p: X \to Y$, we say that a 1-simplex $\ksi \in X$ is \textit{$p$-cocartesian} if
            for any diagram of the form below with $0 \leq k < n$ and $n \geq 2$, there exists a lift as denoted.
            \begin{equation}
                  \label{COCART_EQ}
                  \begin{tikzcd}
                        \Delta^{0,1} \arrow[d, hookrightarrow] \arrow[dr, "\ksi"]
                        \\
                        \Lambda^0[n] \arrow[r] \arrow[d, hookrightarrow]
                        &
                        X \arrow[d, "p"]
                        \\
                        \Delta[n] \arrow[r] \arrow[ur, dashed, "\exists"]
                        &
                        Y
                  \end{tikzcd}
            \end{equation}
            The map $p$ is a \textit{cocartesian fibration} if
            $p$ is an inner fibration and satisfies the analogue the definition of Grothendieck op-fibration:
            for all objects $x \in X$ and arrows $f \colon y \to p(x)$ in $Y_1$,
            there exists a $p$-cocartesian $\hat f \in X_1$ such that $p(\hat f) = f$.

            It is clear that if an arrow $f$ in some category $\mathcal C$ is $p$-cocartesian for $p: \mathcal C \to \mathcal B$,
            then $f \in N'\mathcal C_1$ is $N'(p)$-cocartesian,
            and similarly that if $p$ is a Grothendieck op-fibration, then $N'(p)$ is a cocartesian fibration,
            where $N': \Cat \to \sSet$ is the nerve.

            
      \item Given a map of $\infty$-categories $\mathcal C \to \mathcal D$ and a 0-simplex $d \in D$, denote by $\mathcal C_{\langle d \rangle}$ the pullback below in $\sSet$.
            \begin{equation}
                  \begin{tikzcd}
                        \mathcal C_{\langle d \rangle} \arrow[r] \arrow[d]
                        &
                        \** \arrow[d, "d"]
                        \\
                        \mathcal C \arrow[r]
                        &
                        \mathcal D
                  \end{tikzcd}
            \end{equation}
      \end{itemize}
\end{definition}

\begin{definition}
      \label{INF_OP_DEF}
      An \textit{$\infty$-operad} is an $\infty$-category $\O^\otimes$ equipped with a map of simplicial sets
      $p: \O^\otimes \to N(\Fin_{\**})$
      such that the following three conditions hold:
      \begin{enumerate}[label = (\roman*)]
      \item For all inert maps $f \in \Fin_{\**}(A,B)$,
            and for all objects $x \in \O^{\otimes}_{\langle A \rangle}$, 
            there exists a $p$-cocartesian morphism $\bar f: x \to x'$ lifting $f$.
            In particular, $f$ (and specified choices) induces a functor
            $f_!: \O^\otimes_{\langle A \rangle} \to \O^\otimes_{\langle B \rangle}$.
      \item For all maps $f \in \Fin_{\**}(A, B)$, objects $x \in \O^\otimes_{\langle A \rangle}$ and $y \in \O^{\otimes}_{\langle B \rangle}$,
            and choices of $p$-cocartesian lifts $y \to y_{b}$ for each projection $\pi_b: B_+ \to \set{b}_+$
            from $B$ onto one of its elements $b$,
            the induced Segal map
            \begin{equation}
                  \Map^f_{\O^\otimes}(x,y) \longto \prod_{b \in B} \Map_{\O^\otimes}^{p^{b} f}(x, y_{b})
            \end{equation}
            is a weak equivalence,
            where $\Map^f_{\O^\otimes}(-,-) \into \Map_{\O^\otimes}(-,-)$ is the fiber over $f$.
      \item For all objects $A_+$ and all choices of functors $(\pi_a)_!$, the induced Segal map
            \begin{equation}
                  \O^\otimes_{\langle A \rangle} \longto \prod_{a \in A} \O^\otimes_{\langle a \rangle}
            \end{equation}
            is an equivalence of $\infty$-categories.
      \end{enumerate}

      An arrow of $\infty$-operads is a map of simplicial sets $F: (\O, p) \to (\P, q)$ over $N(\Fin_{\**})$ which sends
      $p$-cocartesian maps to $q$-cocartesian maps.
\end{definition}

\subsection{Operadic nerve}
To complete our motivation, we prove that
there is a faithful functor $N^{\otimes}: \sOp_f \to \Op_{\infty}$, the operadic nerve.
%
%
The first stage of this map is the construction of the category of operators associated to a simplicial operad. 

\begin{definition}[{\cite[2.1.1.22]{Lur17}}]
      \label{CATOP_DEF}
      Given $\O \in \sOp$, we define the \textit{category of operators} associated to $\O$, denoted $\O^{\otimes}$, as the following simplicial category.
      The set of objects is the set of all tuples $(A, (x_a)_{a \in A})$
      with $A \in \Fin_{\**}$ and $(x_a)$ a tuple of colors of $\O$ indexed by $A$.
      Given objects $(A, (x_a))$ and $(B, (y_b))$, define the mapping space by
      \begin{equation}
            \O^\otimes((A, (x_a)), (B, (y_b)))
            =
            \coprod_{f \colon A \to B} \prod_{b \in B}\O((x_a)_{a \in \alpha^{-1}(b)}; y_b).
      \end{equation}
      Composition is as expected:
      given composable arrows $A \xrightarrow{f} B \xrightarrow{g} C$
      and operations
      \[
            \psi_c \in \O(g^{-1}c), \qquad \ksi_b \in \O(f^{-1}b)
      \]
      for all $c \in C$ and $b \in B$, define
      \begin{equation}
            \label{POTIMESU_COMP_EQ}
            (\psi_c)_{c \in C} \circ (\ksi_b)_{b \in B} 
            =
            \left( \psi_c \circ (\ksi_{b})_{b \in \alpha^{-1}(c)} \right)_{c \in C}.
      \end{equation}

\end{definition}

This construction is functorial and faithful.

\begin{definition}
      \label{OPNERVE_DEF}
      Given $\O \in \sOp$, the \textit{operadic nerve} $N^{\otimes}(\O)$ is the simplicial set $N(\O^{\otimes})$,
      where $N: \sCat \to \sSet$ is the homotopy coherent nerve.
      Since $\O^\otimes$ has a canonical map to $\Fin_{\**}$, $N^\otimes(\O)$ has a canonical map to the nerve of $\Fin_{\**}$.
\end{definition}

\begin{proposition}
      [{\cite[Prop. 2.1.1.27]{Lur17}}]
      If $\O$ is a fibrant simplicial operad, then $N^\otimes(\O)$ is an $\infty$-operad.
\end{proposition}

We record the following easy lemma.
\begin{lemma}
      \label{PULLDIS_LEM}
      If $D \in \sSet$ is discrete, then a square with final object $D$
      is a pullback
      iff it is a homotopy pullback
      iff each induced square with final object $\set{d}$ is a pullback.
\end{lemma}

Lastly, we show this construction lifts to a functor.
We could not find a statement or proof of this result in the literature; we record it here for completeness and later reference.
\begin{proposition}
      \label{NOTIMES_FUN_PROP}
      The assignment $\O \mapsto N^\otimes(\O)$ is part of a functor $\sOp_f \to \mathsf{Op}_\infty$
      from fibrant simplicial operads to $\infty$-operads.
\end{proposition}
\begin{proof}
      For any functor $F: \O \to \P$ in $\sOp$, $N^\otimes(F)$ is a map over $N(\Fin_{\**})$
      for the natural maps $N^\otimes(\O) \xrightarrow{p} N(\Fin_{\**})$ and $N^\otimes(\P) \xrightarrow{q} N(\Fin_{\**})$.
      It thus suffices to check that for any functor $F: \O \to \P$,
      $N^\otimes(F)$ sends $p$-cocartesian maps to $q$-cocartesian maps.

      By \cite[Remark 2.1.2.9]{Lur17}, it suffices to check this on $p$-cocartesian morphisms living over
      the projection maps $\pi_r: A_+ \to \langle 1 \rangle$, with $\pi_r(a) = 1$ if $r = a$ and the basepoint 0 of $\langle 1 \rangle$ otherwise.

      To that end, consider the rectangle below,
      where $\sigma_s: \langle 1 \rangle \to B$ sends $1$ to $s \in B$,,
      $\alpha_{r,s}: A \to B$ is $\pi_a \sigma_b$
      (so sends $r$ to $s$ and everything else to the basepoint),
      and $\phi \in \O^\otimes_{\pi_r}((x_a)_{a \in A}; x) = \O(x_r; x)$ is some $p$-cocartesian map over $\pi_r$.
      \begin{equation}
            \begin{tikzcd}
                  \O^{\otimes}_{\sigma_s}(x; (y_b)_{b \in B}) \arrow[r, "\phi^{\**}"] \arrow[d, "F"']
                  &
                  \O^{\otimes}_{\alpha_{r,s}}((x_a)_{a \in A}, (y_b)_{b \in B}) \arrow[d, "F"]
                  \\
                  \P^{\otimes}_{\sigma_s}(F(x); (F(y_b))_{b \in B}) \arrow[r, "{N^\otimes(\phi)^{\**}}"] \arrow[d]
                  &
                  \P^{\otimes}_{\alpha_{r,s}}((F(x_a))_{a \in A}, (F(y_b))_{b \in B}) \arrow[d]
                  \\
                  \set{\sigma_s} \arrow[r, "{(\pi_r)^{\**}}"]
                  &
                  \set{\alpha_{r,s}}
            \end{tikzcd}
      \end{equation}
      But we understand this rectangle explicitly:
      \begin{align*}
        \O^{\otimes}_{\sigma_s}(x; (y_b)_{b \in B})
        &=
          \prod_{b \in B \setminus s} \O(\varnothing; y_b) \times \O(x; y_s)
        \\
        \O^{\otimes}_{\alpha_{r,s}}((x_a)_{a \in A}; (y_b)_{b \in B})
        &=
          \prod_{b \in B \setminus s} \O(\varnothing; y_b) \times \O(x_r; y_s)
      \end{align*}
      and similarly for $\P^{\otimes}$,
      and $\phi^{\**}$ is just pre-composition by the actual operation $\phi \in \O(x_r; x)$.
      Since by \cite[Prop. 2.4.1.10]{Lur09} and Lemma \ref{PULLDIS_LEM} the big rectangle is a pullback,
      $\phi$ (and hence $\phi^{\**}$) must be an isomorphism.
      Thus $F(\phi)$ (and hence $N^\otimes(\phi)^{\**}$) is an isomorphism,
      and hence the lower rectangle is a pullback.
      The reverse directions of Lemma \ref{PULLDIS_LEM} and \cite[Prop. 2.4.1.10]{Lur09} complete the proof.
\end{proof}

\begin{remark}
      \label{PREOP_REM}
      In fact, this gives a functor $N^\otimes(-): \sOp \to \mathsf{PreOp}_\infty$
      from all simplicial operads to the (1)-category of \textit{$\infty$-preoperads} (see \cite[\S 2.1.4]{Lur17}):
      the above functoriality was independent from the fibrancy of $\O$ and $\P$,
      and the construction $N^\otimes(-)$ always gives a preoperad by \textit{loc cit}.
\end{remark}

\section{Equivariant preliminaries}
\label{EPRELIM_SEC}

For the rest of the paper, we fix a finite group $G$.
We will now generalize the definitions found in Sections \ref{CSO_SEC} and \ref{INFTYOP_SEC} to an equivariant context.

\subsection{Finite pointed $G$-sets}

Analogously to how $\infty$-operads live over the category $\Fin_{\**}$ of finite pointed sets,
equivariant $\infty$-operads live over a category of finite pointed $G$-sets.
However, to provide for a complete ``genuine'' equivariant picture, our category of finite $G$-sets also needs to contain all finite $H$-sets as well for every $H \leq G$, as in Definition \ref{FGPT_DEF} below.

\begin{definition}
      \label{FGP_DEF}
      Let $\mathsf F^G$ denote the category of $G$-objects in $\mathsf F$, i.e.
      a fixed full subcategory of all ordered finite $G$-sets with unordered actions of $G$ and unordered $G$-maps,
      such that the wide\footnote{A subcategory is called \textit{wide} if it contains all objects of the original category.} subcategory with ordered $G$-maps is skeletal.
      In particularly, following \cite{GM_presheaves}, we may choose $\Fin^G$ to be the category with objects $(\underline{n}, \alpha)$ with $\alpha: G \to \Sigma_n$ a homomorphism.
      As in Definition \ref{FIN_DEF}, $\mathsf F^G$ has a natural \textit{bipermutative} structure with respect to product and disjoint union.

      Let $\mathsf O_G$ denote the full subcategory of $\mathsf F^G$ of the transitive $G$-sets.   
      A \textit{$\mathcal B$-coefficient system} is a functor $\mathsf O_G^{op} \to \mathcal B$ for some category $\mathcal B$.
\end{definition}

Particularly simple coefficient systems are given by the system of fixed points of a $G$-object $X \in \mathcal B^G$.
Define the $i_{\**}X$ to be the coefficient system
\[
      i_{\**}X(G/H) = X^H := \lim(H \to G \xrightarrow{X} \mathcal B).
\]
If $\mathcal B$ is closed symmetric monoidal with unit $I$, then $X^H \simeq \mathcal B^G(G/H \cdot I, X)$.

\begin{convention}
To avoid confusion following Definition \ref{CATFIB_DEF}, we will specify which type of Grothendieck fibration we mean by name and by identifying the base.
Our single exception will be for coefficient systems of \textit{sets}, 
where we will just write ``coefficient system'' to mean either the presheaf functor
or the associated \textit{\textbf{cartesian}} fibration.
\end{convention}

Now, replacing the role of finite pointed sets from \S \ref{INFTYOP_SEC} will
not just be finite pointed $G$-sets $\mathsf F^G_{\**}$, but a whole coefficient system.

\begin{definition}[\cite{BDGNS,Nar17}]
      \label{FGPT_DEF}
      Let $\underline{\Fin}^G_{\**} \to \mathsf O_G^{op}$ denote the Grothendieck op-fibration
      associated to the functor below
      \[ 
            \mathsf O_G^{op} \longto \Cat,
            \qquad
            R \longmapsto \Fin^G_{\**} \downarrow_+ R_+,
      \]
      where $\Fin^G_{\**} \downarrow_+ R_+$ denotes the full subcategory spanned by arrows for the form $(A \to R)_+$.
\end{definition}

We unpack this definition as follows.
Objects are $G$-maps $A \to R$ with $A \in \Fin^G$ and $R \in \mathsf O_G$,
and an arrow $f: (A \to R) \to (B \to S)$ is given by a triple $(q, A_{f}, \bar f)$
of a $G$-map $q: S \to R$,
an inclusion $A_f \into q^{\**} A$ over $S$,
and a $G$-map $\bar f: A_{f} \to B$ over $S$.
\begin{equation}
      \label{FPGSET_EQ}
      \begin{tikzcd}
            A \arrow[d]
            &
            q^{\**} A \arrow[d] \arrow[l] \arrow[dl, phantom, "\llcorner"'{very near start}]
            &
            A_f \arrow[l, hookrightarrow] \arrow[d] \arrow[r, "\bar f"]
            &
            B \arrow[d]
            \\
            R
            &
            S \arrow[l, "q"'] \arrow[r, equal]
            &
            S \arrow[r, equal]
            &
            S
      \end{tikzcd}
\end{equation}

Composition is given by pullbacks (see \eqref{FINP_COMP_EQ} below),
and we have an obvious map $\underline{\Fin}^G_{\**} \to \mathsf O_G$ sending $(A \to R)$ to $R$.

\begin{remark}
      \label{PTGSET_REM}
      This description above indeed recovers $\underline{\Fin}^G_{\**}$,
      the amalgamation of the categories of pointed finite $H$-sets for all $H \leq G$
      given by Definition \ref{FGPT_DEF}.
      When $R = G/H$, an arrow $A \to G/H$ is equivalent to data of an $H$-set by considering the fiber $A_H$ over $eH \in G/H$.
      Moreover, for $S = G/K$ with $K \leq H$ and $q: G/K \to G/H$ the canonical quotient map,
      the pullback $q^{\**}A \to G/K$ represents the restriction of $A_H$ to a $K$-set, as the fiber over $eK$ is precisely $i^{\**}_K A_H$.
      Finally, a pointed map of $K$-sets $A_{K,+} \to B_{K,+}$ is the same as a partially defined map of $K$-sets $A_K \hookleftarrow (A_f)_K \to B_K$,
      where the orbits of $A_K \setminus (A_f)_K$ are sent to the basepoint of $B_{K,+}$.
      Thus, we should think of \eqref{FPGSET_EQ} as representing a pointed equivariant map from a pointed $H$-set to a pointed $K$-set.
\end{remark}
%


\begin{definition}
      For any object $(A \to R)$, we define the set of \textit{orbits} to be $A/G$,
      where we remember that each element $U \in A/G$ comes with a specified inclusion $U \into A$ over $R$.
\end{definition}

\begin{notation}
      As a general convention, we will use $A,B,C$ to denote arbitrary $G$-sets,
      $U,V,W$ to denote their orbits, and
      $R,S,T$ to denote transitive $G$-sets acting as bases of the objects in $\underline{\Fin}^G_{\**}$.
\end{notation}

Following Definition \ref{FIN_DEF}, we name several classes of maps in $\underline{\Fin}^G_{\**}$.

\begin{definition}
      \label{FING_MAPS_DEF}
      A map $f = (q, A_f, \bar f): A \to B$ in $\underline{\Fin}^G_{\**}$ is called
      \begin{itemize}
      \item \textit{inert} if $\bar f$ is an isomorphism.
      \item a \textit{projection map} if
            $A_f = B = U$ for some $U \in A/G$, and
            both $q$ and $\bar f$ are identities;
            in this case, we denote $f$ by $\pi_U$.            
      \item a \textit{quotient map} if $A_f = q^{\**}A$ and $\bar f$ is an isomorphism.
      \item an \textit{orbit map} if $A = R$, $A_f = q^{\**}A$, and $\bar f$ is the identity.
      \end{itemize}
\end{definition}

\begin{remark}
      Remark \ref{PTGSET_REM} provides the equivalence of the above notions of \textit{inert} and \textit{projection} with those from Definition \ref{FIN_DEF}:
      a map is \textit{inert} (resp. a \textit{projection})
      if the map of pointed $G$-sets over $S$ is surjective and additionally injective away from the basepoint
      (resp. $B = U$ for some orbit $U \in A/G$, $\bar f(V)$ is not the basepoint for all $V \neq U$ in $A/G$, and $\bar f|_U$ is the identity).      
\end{remark}

Let $\underline{\Fin}^G_{\**,in}, \underline{\Sigma}^G \subseteq \underline{\Fin}^G_{\**}$
denote the wide subcategories with inert maps and quotient maps respectively.

\begin{remark}
      \label{LR_REM}
      We note that $\underline{\Sigma}^G$ is actually a subcategory of $\underline{\Fin}^G$ \textit{unpointed} finite $G$-sets
      (and in fact is the maximal $G$-subgroupoid over $\mathsf O_G$).
      Moreover, in the case $G = \**$, $\underline{\Sigma}^G$ is just the \textit{symmetric category} $\Sigma = \amalg \Sigma_n$, the disjoint union of all symmetric groups.
      Keeping with this terminology, we call $\underline{\Sigma}^G$ the \textit{$G$-symmetric category}.
      %
\end{remark}

We end this subsection with a technical look at composition in $\underline{\Fin}^G_{\**}$.
Specifically, given arrows
\begin{equation}
      \label{TWOARROWS_EQ}
      (q, A_f, \bar f): (A \to R) \to (B \to S),
      \qquad
      (p, B_g, \bar g): (B \to S) \to (C \to T),
\end{equation}
define $A_{gf}$ to be the pullback of $A_f$ and $B_g$ over $B$.
These pieces fit together in the following commutative diagram,
where in particular the denoted squares are pullbacks
and we define the composite $g \circ f$ to be the outer rectangle.
\begin{equation}
      \label{FINP_COMP_EQ}
      \begin{tikzcd}
            A \arrow[d, equal]
            &
            p^{\**} q^{\**} A \arrow[d] \arrow[l]
            &
            p^{\**} A_f \arrow[d] \arrow[r, equal] \arrow[l, hookrightarrow]
            &
            p^{\**} A_f \arrow[d, equal]
            &
            A_{gf} \arrow[l, hookrightarrow, dashed] \arrow[d, equal] \arrow[r, "\bar{gf}"]
            &
            C \arrow[dd, equal]
            \\
            A \arrow[dd, equal]
            &
            q^{\**} A \arrow[dd] \arrow[l]
            &
            A_f \arrow[d, "\bar f"'] \arrow[l, hookrightarrow]
            &
            p^{\**} A_f \arrow[d, "p^{\**}\bar f"] \arrow[l, "p"'] \arrow[dl, phantom, "\llcorner"{very near start}]
            &
            A_{gf}  \arrow[l, hookrightarrow, dashed] \arrow[d] \arrow[dl, phantom, "\llcorner"{very near start}]
            \\
            &&
            B \arrow[d]
            &
            p^{\**} B \arrow[l, "p"'] \arrow[d]
            &
            B_g \arrow[l, hookrightarrow] \arrow[d] \arrow[r, "\bar{g}"]
            &
            C \arrow[d]
            \\
            R
            &
            S \arrow[l, "q"'] \arrow[r, equal]
            &
            S
            &
            T \arrow[l, "p"'] \arrow[r, equal]
            &
            T \arrow[r, equal]
            &
            T
      \end{tikzcd}
\end{equation}



We may identify the inverse image of orbits in $C$ under the composite $\bar{gf}$. 

\begin{notation}
      \label{PV_MAP_NOT}
      Given $q: S \to R$ in $\mathsf O_G$ and $(A \to R) \in \underline{\Fin}^G_{\**}$, for all $\bar U \in q^{\**}A/G$ we write
      $p_{\bar U}: \bar U \to p(\bar U)$ for the induced map in $\mathsf O_G$.
\end{notation}

\begin{lemma}
      \label{GFBAR_LEM}
      For arrows $f$ and $g$ as in \eqref{TWOARROWS_EQ}, and any $W \in C/G$, we have 
      \begin{equation}
            \label{PULLASSEM_EQ}
            \bar{gf}^{-1}(W)
            =
            \coprod_{\bar V \in \bar{g}^{-1}(W) / G} p_{\bar V}^{\**}(\bar f^{-1}(p(\bar V))).
      \end{equation}
\end{lemma}
\begin{proof}
      First, we note that for any $V \in B / G$, we have that
      \begin{equation}
            \coprod_{\bar V \in p^{-1}(V)/G} p_{\bar V}^{\**}\left(\bar f^{-1}(V)\right)
            =
            (p^{\**}\bar{f})^{-1}(p^{-1}(V)). 
      \end{equation}
      Indeed, the triple of inclusions $(\bar V \into p^{\**} B$, $V \into B$, $\bar{f}^{-1}(V) \into A_f)$
      induces an inclusion of pullbacks
      $p_{\bar V}^{\**}\left(\bar{f}^{-1}(V)\right) \into p^{\**} A_f$,
      whence we conclude $(p^{\**}\bar f)^{-1}(\bar V) = p_{\bar V}^{\**}(\bar f^{-1}(V))$.
      Second, we in fact have the more general statement that for any $G$-closed subset $S \subseteq (p^{\**} V)/G$,
      \begin{equation}
            \coprod_{\bar V \in S} p_{\bar V}^{\**}\left(\bar{f}^{-1}(V)\right)
            =
            (p^{\**}\bar{f})^{-1}(S).
      \end{equation}
      Finally, \eqref{PULLASSEM_EQ} follows since we have the identifications below for all $W \in C/G$ by \eqref{FINP_COMP_EQ}.
      \begin{align*}
        \bar{gf}^{-1}(W)
        & =
          (p^{\**}\bar{f})^{-1}\left(\bar{g}^{-1}(W) \right)
        \\
        \coprod_{\bar{V} \in \bar{g}^{-1}(W)/G} p_{\bar V}^{\**} \left( \bar{f}^{-1}(p(\bar V)) \right)
        & =
          \coprod_{V \in B/G} \left(
          \coprod_{\bar V \in \left( (p^{-1}(V)) \ \mathbin{\cap} \ (\bar{g}^{-1}(W)) \right) / G} p_{\bar V}^{\**}\left( \bar{f}^{-1}(p(\bar V))\right)
          \right)
      \end{align*}
\end{proof}

\subsection{Equivariant operads}

In this section, we introduce two of the major players of this paper, colored genuine equivariant operads and $\mathsf O_G$-$\infty$-operads.

\subsubsection*{Colored genuine equivariant operads}
We begin with the former. As noted in the introduction,
the category $\sOp_G^{\set{\**}}$ of single-colored genuine equivariant operads were introduced in \cite{BP_geo} as a projective model category Quillen equivalent to
the category $\sOp^G_{\set{\**}}$ of single-colored simplicial equivariant operads,
i.e. operads in simplicial $G$-sets (or equivalently $G$-objects in $\sOp_{\set{\**}}$).
Objects $\P \in \sOp_G$ have evaluations for all finite $H$-sets for all $H \leq G$ ---
in fact, have an underlying functor $\underline{\Sigma}^{G} \to \sSet$ ---
and have composition laws which respect the orbit structures of the various participating $G$-sets.
Morally, these play the same role coefficient system of spaces played in the Elmendorf-Piacenza Theorem \cite{Elm83,Pia91}
showing the Quillen equivalence $\Top^G \simeq_Q \Top^{\mathsf O_G^{op}}$.

Below, we give a description of a generalization of this structure which allows for multiple objects/colors\footnote{
  It would not be wrong to call these structures \textit{genuine equivariant multicategories}.
  However, as with our earlier conventions, we follow Lurie, Berger, Cisinski, and Moerdijk and write ``operad'' to refer to both
  the classical single-colored notion as well as the more general many-colored variety.}.
When working with many-colored equivariant simplicial operads,
the associated set of colors is in fact a $G$-set, with action inherited by the $G$-action on the operad itself (see e.g. \cite{BP_sq}).
However, in the \textit{genuine} setting, we are instead allowed to have a non-trivial \textit{coefficient system} of colors,
agreeing with our moral intuition.

The main ingredient in this many-colored generalization is a replacement of $\underline{\Sigma}^G \to \mathsf O_G^{op}$ with a many-colored variant $\underline{\Sigma}^G_{\mathfrak C}$.
First, we recall the following categories of \textit{tuples}.

\begin{definition}
      \label{FWR_DEF}
      For any category $\mathcal C$, let $\mathsf F \wr \mathcal C \to \Fin$ denote the (split) Grothendieck fibration associated to the functor
      \[
            \mathsf F^{op} \longto \Cat,
            \qquad
            A \mapsto \mathcal C^{\times A}.
      \]
      Explicitly, objects are tuples $(A, (X_a))$ of elements in $\mathcal C$, and arrows $(A, (X_a)) \to (B, (Y_b))$ are
      maps $\alpha: A \to B$ in $\mathsf F$ and arrows $f_a: X_a \to Y_{\alpha(a)}$ in $\mathcal C$.

      We write $\Fin_s \wr \mathcal C$ (resp. $\Sigma \wr \mathcal C$) for the analogous definition
      replacing $\Fin$ with the wide subcategory $\Fin_s$ of surjective maps (resp. $\Sigma$ of isomorphisms).
\end{definition}

\begin{definition}      
      Define the \textit{edge orbit} and \textit{leaf orbit} functors\footnote{
        This terminology comes from recognizing $\underline{\Sigma}^G$ as the category of \textit{$G$-corollas} as in \cite{BP_geo,Per18}.}
      \begin{equation}
            \mathbf E_G: \underline{\Sigma}^{G,op} \longto \Fin_s \wr \mathsf O_G,
            \qquad
            \mathbf L_G: \underline{\Sigma}^{G, op} \longto \Fin_s \wr \mathsf O_G,
      \end{equation}
      by letting $\mathbf E_G(A \to R)$ be the tuple of all orbits $(A/G, (U)) \amalg (\**, R)$,
      and $\mathbf L_G(A \to R) = (A/G, (U))$ the tuple of orbits of the source.
\end{definition}

\begin{definition}
      \label{SGC_DEF}
      Fix a coefficient system of sets $\mathfrak C$.
      The \textit{$\mathfrak C$-colored $G$-symmetric category}, denoted $\underline{\Sigma}^{G,op}_{\mathfrak C}$, is the pullback below.
      \begin{equation}
            \begin{tikzcd}
                  \underline{\Sigma}^{G,op}_{\mathfrak C} \arrow[d] \arrow[r]
                  &
                  \mathsf F_s \wr \mathfrak C \arrow[d]
                  \\
                  \underline{\Sigma}^{G,op} \arrow[r, "\mathbf E_G"]
                  &
                  \mathsf F_s \wr \mathsf O_G
            \end{tikzcd}
      \end{equation}
      
      Objects are called \textit{$\mathfrak C$-signatures}, and are written 
      \[
            \vect{C} = \left(A \to R, ((x_U); x_R) \right) =
            \left( \substack{A \\ \downarrow \\ R},\ ((x_U); x_R) \right)
      \]      
      with
      $(A \to R)$ in $\underline{\Sigma}^G$, $U \in A/G$, and $x_U \in \mathfrak C_U$ and $x_R \in \mathfrak C_R$.
      We call $(A \to R)$ the \textit{arity} of the signature, and will sometimes denote the arity of $\vect{C}$ by $C$.

      Arrows in the \textit{opposite}\footnote{
        This convention is further discussed in Warning \ref{SG_WARN}.
      }
      category $\underline{\Sigma}^G_{\mathfrak C}$
      \begin{equation}
            \label{SGC_MAP_EQ}
            f \colon \left( \substack{A \\ \downarrow \\ R}, ((x_U); x_R) \right) \to
            \left( \substack{B \\ \downarrow \\ S}, ((y_V); y_S) \right)
      \end{equation}
      are given by quotient maps $(q, \bar f): (A \to R) \to (B \to S)$ in $\underline{\Fin}^G_{\**}$
      (with $q: S \to R$ and $\bar f: q^{\**}A \xrightarrow{\cong} B$ as in Defn. \ref{FING_MAPS_DEF}) such that
      \begin{equation}
            \label{CGSEQ_EQ}
            q_{\bar U}^{\**}x_{q(\bar U)} = y_{\bar f(\bar U)}
            \qquad \mbox{and} \qquad
            q_S^{\**}x_{R} = y_S
      \end{equation}
      for all $\bar U \in q^{\**}A/G$ (where we note $q(S) = R$).    

      A \textit{$\mathfrak C$-colored $G$-symmetric sequence} is a functor $\underline{\Sigma}^{G}_{\mathfrak C} \to \sSet$.
\end{definition}

A $\mathfrak C$-colored genuine equivariant operad consists of a $\mathfrak C$-colored $G$-symmetric sequence,
equipped with a ``composition law'' for all appropriately-compatible signatures. 

\begin{definition}
      \label{CCOLL_DEF}
      Given some $\vect{C} = \left( A \to R, ((x_U); x_R) \right) \in \underline{\Sigma}^G_{\mathfrak C}$,
      a \textit{compatible collection} is
      a collection of objects $\vect{D}_U = (B_U \to U, ((x_{U,V}); x_U))$, one for each $U \in A/G$.
      The \textit{composite} of the compatible collection is another object in $\underline{\Sigma}^G_{\mathfrak C}$,
      denoted $\vect{C} \circ (\vect{D_U})$, defined to be
      \[    
            \left( \substack{B \\ \downarrow \\ R}, ((x_{U,V}); x_R) \right), 
            \qquad
            \mbox{with}
            \quad
            B = \coprod_{U \in A/G} B_U.      
      \]      
\end{definition}
 
\begin{definition}[{cf. \cite[Eq. (1.11)]{BP_geo}}]
      A \textit{colored genuine equivariant operad} $\P$ is given by the following data:
      \begin{itemize}
      \item A coefficient system $\mathfrak C = \mathfrak C(\P)$ of colors;
      \item A $\mathfrak C$-colored $G$-symmetric sequence $\P: \underline{\Sigma}^{G}_{\mathfrak C} \to \sSet$;
      \item For all compatible collections $\vect{C}$, $(\vect{D}_U)$ as in Definition \ref{CCOLL_DEF},
            a composition structure map
            \[
                  \mu: \P \left( \varrow{A}{R}, ((x_U); x_R) \right) \times
                  \prod_{U \in A/G} \P \left( \varrow{B_U}{U}, ((x_{U,V}), x_U) \right) \longto
                  \P \left( \varrow{B}{R}, ((x_{U,V}), x_R) \right)
            \]
            where $\left( B \to R, ((x_{U,V}); x_R) \right)$ is the composite of the compatible collection.
      \end{itemize}
      These composition structures are natural in $\underline{\Sigma}^G_{\mathfrak C}$, associative, and unital.
      Spelling out naturality, we have that for any compatible collection as in Definition \ref{CCOLL_DEF} with composite $\vect{E} = \vect{C} \circ (\vect{D}_U)$, and arrows $f = (q, \bar f) \colon (A \to R) \to (C \to S) $ in $\underline{\Sigma}^G_{\mathfrak C}$,
      the diagram
      \begin{equation}
            \label{NAT_EQ}
            \begin{tikzcd}
                  \P \left( \varrow{A}{R}, ((x_U); (x_R)) \right)
                  \times
                  \displaystyle\prod_{U \in A/G} \P \left( \varrow{B_U}{U}, ((x_{U,V}); x_U) \right)
                  \arrow[r, "\mu"] \arrow[d, "{(f, \Delta_q)}"']
                  &
                  \P(\vect{E} )
                  \arrow[d, "f"]
                  \\
                  \P \left( \varrow{C}{S}, ((q_W^{\**}x_{q(W)}); q_S^{\**}x_R) \right)
                  \times
                  \displaystyle\prod_{W \in C/G} \P \left( \varrow{q_W^{\**}B_{q(W)}}{W}, ((q_W^{\**}x_{q(W),V}); q_W^{\**}x_{q(W)} \right) \arrow[r, "\mu"]
                  &
                  \P(\vect{E}')
            \end{tikzcd}
      \end{equation}
      commutes, where $q_W: W \to \bar f^{-1}(W) \to q \left(\bar f^{-1}(W)\right)$ is the induced map on orbits,
      $\Delta_q$ is the ``$q$-twisted diagonal'' 
      \[
            \prod_U \P \left( \varrow{B_U}{U} \right)
            \xrightarrow{\Delta}
            \prod_W \P \left( \varrow{B_{q(W)}}{q(W)} \right)
            \xrightarrow{\Pi q_W}
            \prod_W \P \left( \varrow{q_W^{\**}B_{q(W)}}{W} \right),         
      \]
      and $\vect{E}' = q^{\**}\vect{\C} \circ (q_W^{\**} \vect{D}_U)$ is the composite of the compatible collection written in the bottom row of \eqref{NAT_EQ}, 
      which is naturally isomorphic to $q^{\**} \vect{E}$.

      A \textit{functor} $F: \P \to \P'$ of genuine equivariant operads consists of a map of coefficient systems $F_0: \mathfrak C(\P) \to \mathfrak C(\P')$
      and maps $F(\underline C): \P(\underline C) \to \P'(F(\underline C))$ for all $\mathfrak C$-signatures $\underline C$,
      compatible with the composition structure maps.

      We denote the category of genuine equivariant operads and functors by $\sOp_G$.
\end{definition}

See \cite{BP_geo} for a monadic definition of the single-colored case and further discussion.

\begin{warning}
      \label{SG_WARN}
      We record that some of the notational conventions in the previous definition of genuine equivariant operads are \textit{dual} to those written in \cite{BP_geo}.
      This comes out of \cite{Nar17} and the author having chosen the \textit{opposite} convention for
      which category forms the base of the Grothendieck fibration associated to a functor $\mathsf O_G^{op} \to \Cat$
      (see e.g. Definition \ref{FGPT_DEF}).
      
      Specifically, the category $\Sigma_G^{op}$ from \cite{BP_geo}
      is canonically isomorphic (as a cartesian fibration over $\mathsf O_G^{op}$) to $\underline{\Sigma}^{G}$.
\end{warning}

\begin{notation}
      \label{PC_NOT}
      For $\P \in \sOp_G$ and $(A \to R) \in \underline{\Sigma}^G$, we let $\P(A \to R)$ denote
      \[
            \P\left( \substack{A \\ \downarrow \\ R}\right) := \coprod \P \left( \substack{A \\ \downarrow \\ R},\ ((x_U); x_R)\right)
      \]
      where the disjoint union runs over all possible $\mathfrak C$-signatures in $\underline{\Sigma}^G_{\mathfrak C}$
      with arity $(A \to R) \in \underline{\Sigma}^G$.
\end{notation}

As in the non-equivariant case, we expect there to be a model structure on colored genuine equivariant operads, following \cite{BP_geo, BP_sq}.
For this paper, we will just need the following.
\begin{definition}
      A genuine equivariant operad $\P \in \sOp_G$ is called \textit{locally fibrant} if
      $\P(\vect{C})$ is a Kan complex in $\sSet$ for all $\mathfrak C$-signatures $\vect{C} \in \underline{\Sigma}^G_{\mathfrak C}$.
      We denote the full-subcategory spanned by locally fibrant operads by $\sOp_{G,f} \subseteq \sOp_G$.
\end{definition}


\subsubsection*{$\mathsf O_G$-$\infty$-operads}

For the second player, we follow \cite{Nar17},\cite{BDGNS} to define $\mathsf O_G$-$\infty$-operads as a particular case of parametrized $\infty$-operads.
Parallel to replacing $\Sigma$ with $\underline{\Sigma}^G$,
we replace $N(\Fin_{\**})$ with $N(\underline{\Fin}^G_{\**})$.
Specifying Definition \ref{COCART_DEF} to this case, if $p: X \to N(\underline{\Fin}^G_{\**})$ is a fixed map of simplicial sets,
we refer to $p$-cocartesian morphisms in $X$ as \textit{$G$-cocartesian}.

\begin{definition}[{\cite{BDGNS,Nar17}, cf. Defn. \ref{INF_OP_DEF}}]
      \label{OG_INF_OP_DEF}
      An \textit{$\mathsf O_G$-$\infty$-operad} is an $\infty$-category $\O^\otimes$ equipped with a map of simplicial sets
      $p: \O^\otimes \to N(\underline{\Fin}^G_{\**})$
      such that the following three conditions hold:
      \begin{enumerate}[label = (\roman*)]
      \item For all inert maps $f \in \underline{\Fin}^G_{\**}(A \to R, B \to S)$
            and for all objects $x \in \O^{\otimes}_{\langle A \to R \rangle}$, 
            there exists a $G$-cocartesian morphism $\bar f: x \to x'$ lifting $f$.
            In particular, $f$ (and specified choices) induces a functor
            $f_!: \O^\otimes_{\langle A \to R \rangle} \to \O^\otimes_{\langle B \to S \rangle}$.
      \item For all maps $f \in \underline{\Fin}^G_{\**}(A \to R, B \to S)$, objects $x \in \O^\otimes_{\langle A \to R \rangle}$ and $y \in \O^{\otimes}_{\langle B \to S \rangle}$,
            and choices of $G$-cocartesian lifts $y \to y_{V}$ for each projection $\pi_V \in \underline{\Fin}^G_{\**}(B \to S, V \to S)$ of $B$ onto one of its orbits $V$,
            the induced Segal map
            \begin{equation}
                  \Map^f_{\O^\otimes}(x,y) \longto \prod_{V \in B/G} \Map_{\O^\otimes}^{p^{V} f}(x, y_{V})
            \end{equation}
            is a weak equivalence,
            where $\Map^f_{\O^\otimes}(-,-) \into \Map_\O^\otimes(-,-)$ is the fiber over $f$.
      \item For all objects $(A \to R)$ with set of orbits $\set{U \to R}_{A/G}$, and all choices of functors $(\pi_U)_!$, the induced Segal map
            \begin{equation}
                  \O^\otimes_{\langle A \to R \rangle} \longto \prod_{U \in A/G} \O^\otimes_{\langle U \to R \rangle}
            \end{equation}
            is an equivalence of $\infty$-categories.
      \end{enumerate}
\end{definition}

\begin{remark}
      \label{SEGALTYPE_REM}
      We will call an $\infty$-category $\O^\otimes$ satisfying $(i)$ a \textit{$G$-inert (cocartesian) fibration},
      and those satisfying $(i)$ and $(iii)$ to be of \textit{Segal type}.
\end{remark}




\section{The genuine operadic nerve}
\label{GON_SEC}

In this section, we extend the non-equivariant construction $N^\otimes(-): \sOp_f \to \Op_{\infty}$ to the genuine equivariant setting,
and prove Theorem \ref{THMI}.
As in Section \ref{MOTV_SEC}, this will be the composition of a ``category of operators'' construction followed by the homotopy coherent nerve.

\subsection{Genuine category of operators}

We begin by extending Definition \ref{CATOP_DEF} by again applying the philosophy of replacing $\Fin_{\**}$ with $\underline{\Fin}^G_{\**}$.
We first restrict to the case of a single color.

\begin{definition}
      Let $\P \in \sOp_G$ be a genuine equivariant simplicial operad with a single color.
      We define the \textit{genuine equivariant category of operators} associated to $\P$, denoted $\P^{\otimes}$, as follows.
      The set of objects is precisely $\mathrm{Ob}(\underline{\Fin}^G_{\**})$.
      Given objects $(A \to R)$ and $(B \to S)$ in $\underline{\Fin}^G_{\**}$, define the mapping space
      \begin{equation}
            \P^{\otimes}\left( \substack{A \\ \downarrow \\ R}, \substack{B \\ \downarrow \\ S} \right)
            =
            \coprod_{f \in \underline{\Fin}^G_{\**}(A,B)} \prod_{V \in B/G}
            \P \left( \substack{\bar f^{-1}(V) \\ \downarrow \\ V} \right).
      \end{equation}

      Given composable arrows
      \[
            (A \to R) \xrightarrow{(q, \bar f)} (B \to S) \xrightarrow{(p, \bar g)} (C \to T)
      \]
      and operations
      \[
            \ksi_V \in \P \left( \substack{ \bar f^{-1}(V) \\ \downarrow \\ V} \right),
            \qquad
            \psi_W \in \P \left( \substack{ \bar g^{-1}(W) \\ \downarrow \\ W} \right)
      \]
      for all $V \in B/G$ and $W \in C/G$ respectively,
      the composite is given by
      \begin{equation}
            \label{POTIMES_COMP_EX_EQ}
            (\psi_W)_{W \in C/G} \circ (\ksi_V)_{V \in B/G} 
            =
            \left( \psi_W \circ (p_{\bar V}^{\**}\ksi_{p(\bar V)})_{\bar V \in \bar g^{-1}(W)/G} \right)_{W \in C/G}.
      \end{equation}

      Heuristically, we need to pull back the operations $\ksi_V$ along $p_{\bar V}$ until they line up with the orbits of $C$,
      and then compose as in the non-equivariant case \eqref{POTIMESU_COMP_EQ}.
      Explicitly, this is the composite of the following arrows in $\sSet$:
      \begin{align*}
        \P^{\otimes}(B,C) \times \P^\otimes(A,B)
        & = 
          \left(
          \coprod_{g \in \underline{\Fin}^G_{\**}(B,C)} \prod_{W \in C / G} \P
          \left( \substack{\bar{g}^{-1}(W) \\ \downarrow \\ W} \right)
        \right)
        \times
        \left(
        \coprod_{f \in \underline{\Fin}^G_{\**}(A,B)} \prod_{V \in B/G} \P
        \left( \substack{\bar f^{-1}(V) \\ \downarrow \\ V} \right)
        \right)
        \\
        & =
          \coprod_{(g,f)} \left[
          \prod_{W \in C/G} \P \left( \substack{\bar{g}^{-1}(W) \\ \downarrow \\ W} \right)
        \times
        \prod_{V \in B/G} \P \left( \substack{\bar f^{-1}(V) \\ \downarrow \\ V} \right)
        \right]
        \\
        & \xrightarrow{\Delta}
          \coprod_{(g,f)} \left[
          \prod_{W \in C/G} \left(
          \P \left( \substack{\bar{g}^{-1}(W) \\ \downarrow \\ W} \right)
        \times
        \prod_{\bar{V} \in \bar{g}^{-1}(W) / G} \P \left( \substack{ \bar f^{-1}(p(\bar V)) \\ \downarrow \\ p(\bar V)} \right)
        \right)
        \right]
        \\
        & \xrightarrow{p_{\bar V}}
          \coprod_{(g,f)} \left[
          \prod_{W \in C / G} \left(
          \P \left( \substack{\bar{g}^{-1}(W) \\ \downarrow \\ W} \right)
        \times
        \prod_{\bar V \in \bar{g}^{-1}(W) / G} \P \left( \substack{ p_{\bar V}^{\**}\bar f^{-1}(p(\bar V)) \\ \downarrow \\ \bar V} \right)
        \right)
        \right]
        \\
        & \xrightarrow{\mu} 
          \coprod_{(g,f)}
          \prod_{W \in C/G} \P \left( \substack{ \bar{gf}^{-1}(W) \\ \downarrow \\ W} \right)
        \\
        & \into        
          \coprod_{h \in \underline{\Fin}^G_{\**}(A,C)}
          \prod_{W \in C/G} \P \left( \substack{ \bar{h}^{-1}(W) \\ \downarrow \\ W} \right),
        \stepcounter{equation}\tag{\theequation}\label{POTIMES_COMP_EQ}
      \end{align*}
      where $\mu$ is the genuine operadic composition map, and
      $p_{\bar V}$ denotes the map $\bar V \to p(\bar V)$ in $\mathsf O_G$ as well as the
      associated cartesian arrow in $\underline{\Sigma}^{G}$.\\


      Now, let $\P$ be an \textit{arbitrary} genuine equivariant simplicial operad,
      with coefficient system of colors $\mathfrak C$.
      The \textit{genuine equivariant category of operators} associated to $\P$, denoted $\P^\otimes$, is defined as follows.
      The set of objects is the set of equivariant tuples\footnote{
        Following Definition \ref{SGW_DEF}, this will be the set of objects of the category denoted $\Sigma_G \wr \mathfrak C$.}
      \[
            \left( \substack{A \\ \downarrow \\ R}, (x_U)_{U \in A/G} \right)
      \]
      with $x_U \in \mathfrak C_U$
      (compared to $\mathfrak C$-signatures, only the ``input'' orbits are labeled).

      Given such tuples $(A \to R,(x_U))$ and $(B \to S,(y_V))$, we define the mapping spaces
      \begin{equation}
            \label{POTIMES_MAP_EQ}
            \P^{\otimes}\left( \left( \varrow{A}{R}, (x_U) \right), \left( \varrow{B}{S}, (y_V) \right) \right)
            = 
            \coprod_{f \in \underline{\Fin}^G_{\**}(C,D)}
            \prod_{V \in L_G(D)}
            \P\left(\substack{\bar f^{-1}(V) \\ \downarrow \\ V}, ((q_{\bar U}^{\**}x_{q(\bar U)})_{\bar U \in \bar f^{-1}(V)}; y_V)\right).
      \end{equation}
      Composition is defined analogously as in \eqref{POTIMES_COMP_EQ},
      by using the naturality of $\P$ with respect to quotient maps in $\underline{\Sigma}^G$ as in \eqref{PBARV_EQ},
      for $\bar V \in p^{\**}B/G$ and $\bar U \in q^{\**}A/G$ (cf. \eqref{FINP_COMP_EQ}).
      \begin{align*}
        \left( \substack{\bar f^{-1}(p(\bar V)) \\ \downarrow \\ p(\bar V)} \right)
        & \xrightarrow{f = (p_{\bar V}, id)}
          \left( \substack{p_{\bar V}^{\**}(\bar f^{-1}(p(\bar V))) \\ \downarrow \\ \bar V} \right),
        \stepcounter{equation}\tag{\theequation}\label{PBARV_EQ}
        \\
        \P\left(\substack{\bar f^{-1}(p(\bar V)) \\ \downarrow \\ p(\bar V)}, ((q_{\bar U}^{\**}x_{q(\bar U)}) 
        ; y_{p(\bar V))}\right)
        & \xrightarrow{p_{\bar V}}
          \P\left(\substack{p_{\bar V}^{\**}\left(\bar f^{-1}(p(\bar V))\right) \\ \downarrow \\ \bar V}, \left(p_{\bar V}^{\**}\left(q_{\bar U}^{\**}X_{q(\bar U)}\right); p_{\bar V}^{\**}Y_{p(\bar V)}\right)\right),
      \end{align*}
      where we observe the following.
      \[
            p_{\bar V}^{\**}\left( \left( q_{\bar U}^{\**}x_{q(\bar U)}\right)_{\bar U \in \bar f^{-1}(p(\bar V))} \right)
            =
            \left(p_{\ddot{U}}^{\**}q_{p(\ddot{U})}^{\**}x_{qp(\ddot{U})}\right)_{\ddot{U} \in p_{\bar V}^{\**}(\bar f^{-1}(p(\bar V)))}
            =
            \left((qp)_{\ddot{U}}^{\**}x_{qp(\ddot{U})}\right)_{\ddot{U}}
      \]
\end{definition}

\begin{convention}
      The results in this section about $\P \in \sOp_G$ will have proofs which only speak to the case where $\P$ has a single color.
      The methods can be carried through without issue --- beyond excessive bookkeeping --- into the many-colored setting (following \eqref{CGSEQ_EQ}).
\end{convention}

The following is the main result of this subsection.
\begin{proposition}
      \label{POTIMES_SCAT_PROP}
      For $\P \in \sOp_G$, $\P^{\otimes}$ is a simplicial category, and the construction extends to a functor
      $(-)^{\otimes}: \sOp_G \to \sCat$.
\end{proposition}
\begin{proof}
      It remains to check associativity, unitality, and functoriality.
      
      The identity on an object $A$ is given by the identity map in $\underline{\Fin}^G_{\**}(A,A)$ along with the identity in each
      \mbox{$\P(U \to U)$}, $U \in A/G$,
      and hence unitality of $\P^\otimes$ follows from the naturality of the unitality of $\P$ with respect to orbit maps.
      
      Associativity will follow from the associativity of cartesian lifts in split Grothendieck fibrations and the associativity of $\P$.
      Specifically, given arrows
      \[
            \left( \substack{A \\ \downarrow \\ R} \right)
            \xrightarrow{f = (q, \bar f)}
            \left( \substack{B \\ \downarrow \\ S} \right)
            \xrightarrow{g = (p, \bar g)}
            \left( \substack{C \\ \downarrow \\ T} \right)
            \xrightarrow{h = (r, \bar h)}
            \left( \substack{D \\ \downarrow \\ M} \right),
      \]
      in $\underline{\Fin}^G_{\**}$, \eqref{PULLASSEM_EQ} implies
      \begin{equation}
            \bar{hg}^{\-1}(Q) = \coprod_{\bar W \in \bar h^{-1}(Q)/G} r_{\bar W}^{\**}(\bar g^{-1}(r(\bar w)))
      \end{equation}
      for all $Q \in D/G$, 
      and hence for all $\ddot{V} \in \bar{hg}^{-1}(Q) / G \subseteq r^{\**}B_g / G$, the following triangle commutes.
      \begin{equation}
            \label{TOOMANYEQ_EQ}
            \begin{tikzcd}[column sep = small]
                  \ddot{V} \arrow[rr, "r_{\ddot{V}}"] \arrow[dr, "{(pr)_{\ddot{V}}}"']
                  &&
                  r(\ddot{V}) \arrow[dl, "{p_{p(\ddot{V})}}"]
                  \\
                  &
                  pr(\ddot{V})
            \end{tikzcd}
      \end{equation}
      Thus, for each $Q \in D/G$, we have a factorization of $(rp)_{\ddot{V}} \circ \Delta$ as
      
      \begin{equation}
            \begin{tikzpicture}[baseline= (a).base]
                  \node[scale=.9] (a) at (0,0){
                    \begin{tikzcd}
                          &
                          \displaystyle{
                            \prod\limits_{V \in B/G}
                            \P \left( \substack{ \bar{f}^{-1}(V) \\ \downarrow \\ V} \right)
                          }
                          \arrow[dr, "\Delta"] \arrow[dl, "\Delta"']
                          \\
                          \displaystyle{
                            \prod\limits_{\substack{ {\bar W \in \bar h^{-1}(Q)/G,} \\ {\bar V \in \bar{g}^{-1}(r(\bar W))/G} }}
                            \P \left( \substack{ \bar{f}^{-1}(p(\bar V)) \\ \downarrow \\ p(\bar V)} \right)
                          }
                          \arrow[rr, "\Delta"] \arrow[d, "p_{\bar V}"]
                          &&
                          \displaystyle{
                            \prod\limits_{\substack{ {\bar W}, \\ {\ddot{V} \in r_{\bar W}^{\**}\bar{g}^{-1}(r(\bar W))/G} }}
                            \P \left( \substack{ \bar{f}^{-1}(pr(\ddot{V})) \\ \downarrow \\ pr(\ddot{V})} \right)
                          }
                          \arrow[d, "{(rp)_{\ddot{V}}}"]
                          \\
                          \displaystyle{
                            \prod\limits_{\bar W, \bar V}
                            \P \left( \substack{ p_{\bar V}^{\**}\bar{f}^{-1}(p(\bar V)) \\ \downarrow \\ \bar V} \right)
                          }
                          \arrow[r, "\Delta"]
                          &
                          \displaystyle{
                            \prod\limits_{\bar W, \ddot{V}}
                            \P \left( \substack{ p_{r(\ddot{V})}^{\**}\bar{f}^{-1}(pr(\ddot{V})) \\ \downarrow \\ pr(\ddot{V})} \right)
                          }
                          \arrow[r, "r_{\ddot{V}}"]
                          &
                          \displaystyle{
                            \prod\limits_{\bar W, \ddot{V}}
                            \P \left( \substack{ (rp)_{\ddot{V}}^{\**}\bar{f}^{-1}(pr(\ddot{V})) \\ \downarrow \\ \ddot{V}} \right).
                          }
                    \end{tikzcd}
                  };
            \end{tikzpicture}
      \end{equation}

Hence, by the naturality of the multiplication $\mu$ in our genuine equivariant operad $\P$ with respect to quotient maps,
      either order of the iterated composition factors through the simplicial set
      \begin{equation}
            \label{PRODQ_EQ}
            \prod_{Q \in D/G} \left(
                  \P \left( \substack{ \bar h^{-1}(Q) \\  \downarrow \\ Q} \right)
                  \times
                  \prod_{\bar W \in \bar h^{-1}(Q) / G} \P \left( \substack{ r_{\bar W}^{\**}\bar{g}^{-1}(r(\bar W)) \\ \downarrow \\ V} \right)
                  \times
                  \prod_{\ddot{V} \in \bar{hg}^{-1}(Q)/G} \P \left( \substack{ (pr)_{\ddot{V}}^{\**} \bar{f}^{-1}(pr(\ddot{V})) \\ \downarrow \\ \ddot{V}} \right)
            \right),
      \end{equation}
      and thus associativity of $\P^{\otimes}$ follows from associativity of $\P$.

      Lastly, functoriality is immediate, as maps of genuine equivariant operads are natural with respect to maps in $\mathsf O_G$ and preserve multiplication.
\end{proof}

\begin{example}
      \label{COMM_EX}
      Let $\mathsf{Comm} \in \sOp_G$ be the terminal operad $\mathsf{Comm}(-) = \**$.
      Then the associated genuine category of operators $\mathsf{Comm}^{\otimes}$ is simply all of $\underline{\Fin}^G_{\**}$,
      generalizing \cite[Example 2.1.1.18]{Lur17}.
\end{example}

\subsection{Proof of Theorem \ref{THMI}} 
\label{PROOFI_SEC}

As indicated previously, we make the following definition.

\begin{definition}[cf. Definition \ref{OPNERVE_DEF}]
      Given $\P \in \sOp_G$, the \textit{genuine operadic nerve} of $\P$, denoted $N^{\otimes}\P$,
      is the homotopy coherent nerve of the genuine category of operators
      \[
            N^{\otimes} \P = N(\P^{\otimes}).
      \]
\end{definition}

To prove Theorem \ref{THMI}, we now need to show that
$N^\otimes \P$ is an $\mathsf O_G$-$\infty$-operad whenever $\P$ is locally fibrant, and that
$N^\otimes$ extends to a functor, sending maps of genuine operads to maps of $\mathsf O_G$-$\infty$-operads.
We take care of the first requirement now, extending \cite[Prop. 2.1.1.26]{Lur17}.

\begin{theorem}
      \label{LVLFIB_INFTYOP_THM}
      If $\P \in \sOp_G$ is locally fibrant, then $N^{\otimes}(\P)$ is a $\mathsf O_G$-$\infty$-operad.
\end{theorem}
\begin{proof}
      Since $\P$ is locally fibrant, $\P^{\otimes}$ is fibrant in $\mathsf{sCat}$ (as Kan complexes are closed under products and coproducts),
      and hence $N^{\otimes}(\P)$ is an $\infty$-category.
      Moreover, $\P^\otimes$ has an obvious forgetful functor to $\underline{\Fin}^G_{\**}$
      (induced by the map $\P \to \Comm$) which is a local Kan fibration,
      and thus by \cite[Prop. 2.3.1.5]{Lur09} $p: N^\otimes(\P) \to N(\underline{\Fin}^G_{\**})$ is an inner fibration.

      Now, for all inert maps $f: A \to B$ in $\underline{\Fin}^G_{\**}$,
      we have a canonical map in $\P^{\otimes}(A,B)$, given by isomorphisms in each component,
      which we identify with a 1-simplex $\hat f$ in $N^{\otimes}(\P)$ lying over $f$.
      By \cite[Prop. 2.4.1.10]{Lur09}, $\hat f$ is $G$-cocartesian,
      and hence $(i)$ is satisfied.
      
      For $(ii)$, we note in particular that for all $B$ and all orbits $V \in B/G$,
      there exist $G$-cocartesian $\hat f_V \in \P^{\otimes}(B,V)$ over the projection $\pi_V: B \to V$ in $\underline{\Fin}^G_{\**}(B,V)$.
      We must show that for all maps $g: A \to B$ in $\underline{\Fin}^G_{\**}$, the product of canonical post-composition maps
      \begin{equation}
            \Map_{\P^\otimes}^g(A,B) \longto \prod_{V \in B/G} \Map_{\P^\otimes}^{\pi_{V} g}(A, V)
      \end{equation}
      is a weak equivalence. In fact, it is clear that this map is an isomorphism.

      Finally, we need to show that for all objects $(A \to R)$ in $\underline{\Fin}^G_{\**}$, the induced map
      \begin{equation}
            N^\otimes(\P) \times_{N(\underline{\Fin}^G_{\**})}\set{A \to R}
            \to
            \prod_{U \in A/G} N^{\otimes}(\P) \times_{N(\underline{\Fin}^G_{\**})}\set{U \to R}
      \end{equation}
      is an equivalence.
      However, this is again an isomorphism.
      First, we note that for any $G$-orbit $U$ and object $U \to R$ in $\underline{\Fin}^G_{\**}$,
      the simplicial category $\P^{\otimes}_{\langle U \to R \rangle}$ 
      has a single object $(U \to R)$
      with mapping space $\P(U \xrightarrow{=} U)$.
      More generally, for any object in $\underline{\Fin}^G_{\**}$ of the form $(A \amalg B \to R)$,
      $\P^\otimes_{\langle A \amalg B \to R \rangle}$ also has a single object $(A \amalg B \to R)$,
      with mapping space
      \begin{equation}
            \P^\otimes_{\langle A \amalg B \to R \rangle}(A \amalg B, A \amalg B)
            = \prod_{U \in A/G} \P \left( \substack{U \\ \downarrow \\ U} \right) \times \prod_{V \in B/G} \P \left( \substack{ V \\ \downarrow \\ V} \right)
            = \P^\otimes_{\langle A \to R \rangle}(A,A) \times \P^{\otimes}_{\langle B \to R \rangle}(B,B).
      \end{equation}
      The result then follows as the homotopy coherent nerve $N$ preserves pullbacks and products.
\end{proof}

\begin{remark}
      Following Remark \ref{PREOP_REM}, there is a notion of a \textit{pre-$\mathsf O_G$-$\infty$-operad}
      consisting of marked simplicial sets over $\underline{\Fin}^G_{\**}$ marked with inert morphisms.
      Analogously to the non-equivariant case, the above proof shows that $N^\otimes$ is a functor
      $\sOp_G \to \mathsf{PreOp}_{\infty,G}$.
\end{remark}

The first main theorem now follows.

\begin{proof}
      [Proof of Theorem \ref{THMI}]  
      It remains to show functoriality.
      As \cite[Remark 2.1.2.9]{Lur17} naturally generalizes in the $\mathsf O_G$-$\infty$ setting
      to say that a map preserves all inert maps if and only if it preserves all inert maps over the projection maps $\pi_{V}: B \to V$,
      functoriality follows exactly as in the proof of Proposition \ref{NOTIMES_FUN_PROP}.
\end{proof}


\section{Genuine operadic op-fibrations}
\label{GENOPFIB_SEC}

In this section, we prove Theorem \ref{THMII} 
about a specialization of the functor $N^\otimes$.
Here, we twice extend the work of \cite{Her00} and \cite{Heu} --- once each for equivariance and simplicial enrichment ---
to define genuine operadic op-fibrations in $\sOp_G$ in Section \ref{GOF_SEC},
a generalization of Grothendieck op-fibrations of categories. 
Section \ref{GSYM_SEC} then recalls an appropriate notion of ``symmetric monoidal (simplicial) category'' in this genuine equivariant context, dubbed (simplicial) $E \Sigma_G$-algebra,
and shows that there is a faithful inclusion generalizing $\mathsf{SymMon} \into \mathsf{Op}$.
The remaining two subsections finish the proofs of Theorems \ref{THMII} and \ref{THMIII},
by identifying the image of simplicial $E \Sigma_G$-algebras in genuine equivariant operads
and showing that the genuine operadic nerve sends this notion of symmetric monoidal category to the $\mathsf O_G$-$\infty$-categorical version.

\subsection{Genuine operadic op-fibrations}
\label{GOF_SEC}

In this subsection, we define genuine operadic op-fibrations.

\begin{definition}
      \label{FCOCART_DEF}
      Let $F \colon \mathcal P \to \mathcal Q$ be a map in $\sOp_G$, $C \in \underline{\Sigma}^G$, and $\ksi \in \mathcal P(A \to R, ((x_U); x_R))$.
      The operation $\ksi$ called \textit{level $F$-cocartesian} if,
      for every compatible $(B \amalg R \to S) \in \underline{\Sigma}^G$, $Y_V \in \mathfrak C_V$ for each $V \in B/G$, and $Y_S \in \mathfrak C_S$,
      the commuting diagram below is a strict pullback of simplicial sets.
      \begin{equation}
            \label{COCARTPULL_EQ}
            \begin{tikzcd}
                  \P \left( \substack{B \amalg R \\ \downarrow \\ S},\ \Big( \big((y_V), x_R \big) ; y_S \Big) \right) \arrow[r, "\ksi^{\**}"] \arrow[d, "F"']
                  &
                  \P \left( \substack{B \amalg A \\ \downarrow \\ S},\ \Big( \big((y_V), (x_U) \big) ; y_S \Big) \right) \arrow[d, "F"]
                  \\
                  \mathcal Q \left( \substack{B \amalg R \\ \downarrow \\ S},\ \Big( \big( (F(y_V)), F(x_R) \big) ; F(y_S) \Big) \right) \arrow[r, "{F(\ksi)^{\**}}"]
                  &
                  \mathcal Q \left( \substack{B \amalg A \\ \downarrow \\ S},\ \Big( \big( (F(y_V)), (F(x_U)) \big) ; F(y_S) \Big) \right) 
            \end{tikzcd}
      \end{equation}
      The operation $\ksi$ is called \textit{$F$-cocartesian} if it is level $F$-cocartesian and
      additionally for any $q: (B \to S) \to (A \to R)$ in $\underline{\Sigma}^G$, $q^{\**}\ksi$ is level $F$-cocartesian.
      
      If $F: \mathcal P \to \Comm$ is the unique map to the terminal genuine operad, we refer to $F$-cocartesian operations simply as \textit{cocartesian}.
\end{definition}

\begin{remark}
      We make several remarks.
      \begin{itemize}
      \item We are being slightly cavalier with the ordering of the input colors and source $G$-sets.
            However, as it is clear that $\ksi \in \P(A \to R)$ is $F$-cocartesian iff $\sigma \cdot \ksi$ is $F$-cocartesian for any isomorphism $\sigma$ in $\underline{\Sigma}^G$,
            we will often omit these distinctions.

      \item We will mostly be restricting to working with $F$-cocartesian operations when $F\colon \P \to \mathcal Q$ is a \textit{local fibration}, i.e. each
            \[
                  \P(\vect{C}) \to \mathcal Q(F\vect{C})
            \]
            is a Kan fibration in $\sSet$ for all $\mathfrak C(\P)$-signatures $\vect{C}$.
            In particular, any map between locally discrete genuine operads is a local fibration.

            If $F$ is a local fibration, then \eqref{COCARTPULL_EQ} is a pullback iff it is a homotopy pullback.
            As such, in this restricted setting, we can bypass defining the more ``homotopical'' notion of an \textit{$F$-$h$-cocartesian} operation, where we instead require \eqref{COCARTPULL_EQ} to be a homotopy pullback.
            
      \item We can repackage our definition of $F$-($h$)-cocartesian to be of the form in \eqref{COCART_EQ} (and \cite[Defn. 1.3.1]{Heu})
            if we use the combinatorics of the \textit{genuine $G$-trees} $\Omega_G$ and \textit{genuine equivariant dendroidal sets} $\dSet_G$ from \cite{Per18,BP_sq},
            and an appropriate enhancement of the homotopy coherent nerve to a functor $N: \sOp_G \to \dSet_G$:
            The notion of level $F$-$h$-cocartesian can be captured by a similar $\ksi$-restricted outer horn lifting condition on the map $N(F)$ (cf. \cite[Lemma 2.4.1.10(ii)]{Lur09}).
            \begin{equation}
                  \begin{tikzcd}
                        \Omega_G[C] \arrow[dr, "\ksi"] \arrow[d]
                        \\
                        \Lambda^C[T] \arrow[r] \arrow[d]
                        &
                        N(\P) \arrow[d, "F"]
                        \\
                        \Omega[T] \arrow[r] \arrow[ur, dashed, "\exists"]
                        &
                        N(\mathcal Q)
                  \end{tikzcd}
            \end{equation}
            However, this perspective, while meaningful, will not play a large role in the proofs that follow.
            Thus for the sake of brevity and continuity, we will not elaborate on this description.
      \end{itemize}
      %
\end{remark}

We collect several results about $F$-cocartesian operations.
\begin{lemma}
      \label{GOH_COCART_LEM}
      Let $F: \P \to \mathcal Q$ be a map between genuine equivariant operads. The following hold:
      \begin{enumerate}[label = (\roman*)]
      \item An $F$-cocartesian operation $\ksi \in \P(R \xrightarrow{=} R, (x;y))$ is an isomorphism
            iff its image in $\mathcal Q$ is an isomorphism.
      \item Sequential composites of $F$-cocartesian arrows are $F$-cocartesian:
            If the operations
            \[
                  \ksi \in \P \left( \substack{A \amalg U_0 \\ \downarrow \\ R},\ \Big( \big((x_U), x_{U_0} \big); x_R \Big) \right)
                  \qquad \mbox{and}\qquad
                  \psi \in \P \left( \substack{B \\ \downarrow \\ U_0},\ \Big( \big( x_V \big); x_{U_0} \Big) \right)
            \]
            are $F$-cocartesian, then so is their composite
            \[
                  \psi \circ_{U_0} \ksi \in \P \left( \substack{B \amalg A \\ \downarrow \\ R},\ \Big( \big( (x_V), (x_U) \big); x_R \Big) \right).
            \]
      \item Parallel composites of $F$-cocartesian arrows have a similar universal property:
            If the operations
            \[
                  \ksi_R \in \P \left( \substack{A \\ \downarrow \\ R},\ \Big( \big(x_{R,U}\big); x_R \Big) \right) 
            \]
            are all $F$-cocartesian for some collection of objects $(A \to R) \in \underline{\Sigma}^G$,
            then for any compatible $(B \amalg \coprod R \to S) \in \underline{\Sigma}^G$, $y_V \in \mathfrak C(\P)_V$ for each $V \in B/G$, and $y_S \in \mathfrak C(\P)_S$,
            the diagram below is a strict pullback of simplicial sets.
            \begin{equation}
                  \label{MANYCOCART_EQ}
                  \begin{tikzpicture}[baseline= (a).base]
                        \node[scale=.9] (a) at (0,0){
                          \begin{tikzcd}
                                \P \left( \substack{B \amalg \coprod R \\ \downarrow \\ S},\ \Big( \big((y_V), (x_R) \big) ; y_S \Big) \right) \arrow[r, "{(\ksi_R)^{\**}}"] \arrow[d, "F"']
                                &
                                \P \left( \substack{B \amalg \coprod A \\ \downarrow \\ S},\ \Big( \big((y_V), ((x_{R,U})) \big) ; y_S \Big) \right) \arrow[d, "F"]
                                \\
                                \mathcal Q \left( \substack{B \amalg \coprod R \\ \downarrow \\ S},\ \Big( \big( (F(y_V)), (F(x_R)) \big) ; F(y_S) \Big) \right) \arrow[r, "{(F(\ksi_R))^{\**}}"]
                                &
                                \mathcal Q \left( \substack{B \amalg \coprod A \\ \downarrow \\ S},\ \Big( \big( (F(y_V)), ((F(x_{R,U}))) \big) ; F(y_S) \Big) \right)
                          \end{tikzcd}
                        };
                    \end{tikzpicture}
              \end{equation}
      \end{enumerate}
\end{lemma}

\begin{definition}
      \label{GOF_DEF}
      A map $F: \P \to \mathcal Q$ is called a \textit{genuine operadic op-fibration}
      if $F$ is a local fibration with cocartesian lifts:
      for any arity
      $(A \to R) \in \underline{\Sigma}^G$,
      sources $x_U \in \mathfrak C(\P)_U$ for each $U \in A/G$,
      and operation $\psi \in \mathcal Q(A \to R,\ ((F(x_U)); y_R))$,
      there exists $F$-cocartesian  $\ksi \in \P(A \to R, ((x_U); x_R))$ such that $F(\ksi) = \psi$.
      
      $F$ is additionally \textit{$q$-split} if
      we have a \textit{chosen} system of cocartesian lifts, natural in $\underline{\Sigma}^G$;
      that is, fixed choices of colors and cocartesian arrows
      \[
            (x_U)^{\otimes (A \to R)} \in \mathfrak C(\P)_R,
            \qquad
            \ksi_{(x_U)} \in \P\left( \substack{A \\ \downarrow \\ R},\  \Big( (x_U); (x_U)^{\otimes (A \to R)} \Big) \right),
      \]
      such that for any arrow $q \colon (B \to S) \to (A \to R)$ in $\underline{\Sigma}^G$, we have $\ksi_{q^{\**}(x_U)} = q^{\**}\ksi_{(x_U)}$.
  
      Lastly, if additionally the composite of chosen cocartesian arrows is again a chosen cocartesian arrow,
      then $F$ is called \textit{fully split}.     
\end{definition}

\begin{definition}
      \label{FIBP_DEF}
      $\P \in \sOp_G$ is \textit{($q$-split, fully split) op-fibrant} if the unique map to the terminal genuine equivariant operad $\mathsf{Comm}$ is a ($q$-split, fully split) genuine operadic op-fibration \footnote{
        This is a significantly stronger notion of fibrant than what is required to model genuine equivariant higher algebra:
        We expect there to be a (projective) model structure on $\sOp_G$ Quillen equivalent to the model structures on $\sOp^G$,$\dSet^G$, etc. from \cite{Per18,BP_sq},
        where $\P$ is fibrant iff $\P$ is locally fibrant.
      }.
\end{definition}

\begin{definition}
      \label{FIBQCAT_DEF}
      Given two $q$-split genuine operadic op-fibrations $\P, \P' \to \mathcal Q$ over the same base,
      we say a functor $F: \P \to \P'$ is a \textit{map of op-fibrations}
      if $F$ preserves cocartesian arrows.
      
      Following Definition \ref{CATFIB_DEF}, let $\Fib^q(\mathcal Q) \subseteq \Fib^f(\mathcal Q)$
      denote the subcategories of $\Op_G \downarrow \mathcal Q$ spanned by $q$-split and fully-split operadic op-fibrations over $\mathcal Q$, respectively,
      with maps of op-fibrations.
\end{definition}

\begin{lemma}
      \label{PRECOCART_LEM}
      Suppose $\P$ is op-fibrant. Then an operation $\ksi \in \P(A \to R, ((x_U); x_R))$ is cocartesian iff
      the map 
      \[
            \P \left( \substack{R \\ \downarrow \\ S},\ (x_R; y_S) \right) \xrightarrow{\ksi^{\**}}
            \P \left( \substack{A \\ \downarrow \\ S},\ ((x_U); y_S) \right) 
      \]
      is an isomorphism for all $S$ and $y_S \in \mathfrak C_S$.
\end{lemma}
\begin{proof}
      Given $(B \amalg R \to S)$ and $(y_V)$ as in \eqref{COCARTPULL_EQ}, let
      $\psi$ be a cocartesian arrow with source $((y_v),x_R)$ and arity $(B \amalg R \to S)$.
      The result then follows from Lemma \ref{GOH_COCART_LEM}(ii) and 2-out-of-3 for isomorphisms.
\end{proof}

We end this subsection by comparing the above notion with the original 1-categorical notions.

\begin{remark}
      \label{CATFIB_REM}
      When $G = \**$, a functor $p \colon \P \to \P'$ between \textit{discrete} operads is a genuine operadic op-fibration iff it is an operadic fibration in the sense of \cite{Heu}.
      If $\P$ and $\P'$ are in fact categories, then
      the notions of $p$-cocartesian and ($q$-split, fully-split) genuine operadic op-fibrations
      correspond to the notions of
      $p$-cocartesian and ($q$-split, fully-split) Grothendieck op-fibrations as in Definition \ref{CATFIB_DEF}.
\end{remark}

\subsection{Genuine equivariant symmetric monoidal categories}
\label{GSYM_SEC}

We now quickly recall the main definitions from \cite{BPGSym},
namely a model for genuine equivariant symmetric monoidal categories.
Further details, discussions, and examples, as well as comparisons to other models, can be found there.

A classic symmetric monoidal structure on a category $\V$ encodes a way to multiply elements of $\V$ together.
In particular\footnote{Here, we are using the ``unbiased'' definition of symmetric monoidal category, following e.g. \cite{Lei04}.
}, for any tuple $(x_1,\dots,x_n)$ of objects of $\V$, there is an associated object $\otimes_i x_i$.
For any bicomplete closed symmetric monoidal category $\V$, in particular $\V = \sSet$, there is also a $\V$-enriched notion,
where the multiplication map (and associated natural transformations) are required to be $\V$-enriched.

Equivariantly. we make a similar definition, but starting with a
\textit{coefficient system} of simplicially enriched categories.
By Definition \ref{SCAT_COCART_REM}, this is equivalent to
a split simplicial Grothendieck fibration $\UV \to \mathsf O_G$.
First, we need to define an appropriate notion of ``tuple'' in this context.
Our choice is the following, generalizing Definition \ref{FWR_DEF}.

\begin{definition}
      \label{SGW_DEF}
      Given a set or category $\underline{\mathcal C} \to \mathsf O_G$ over the orbit category, define $\Sigma_G \wr \underline{\mathcal C}$ to be the pullback
      \begin{equation}
            \begin{tikzcd}
                  \Sigma_G \wr \underline{\mathcal C} \arrow[r] \arrow[d]
                  &
                  \mathsf F_s \wr \underline{\mathcal C} \arrow[d]
                  \\
                  \underline{\Sigma}^{G,op} \arrow[r, "\mathbf{L}_G"]
                  &
                  \mathsf F_s \wr \mathsf O_G.
            \end{tikzcd}
      \end{equation}
      Objects are tuples $(A \to R, (x_U))$ with $(A \to R) \in \underline{\Sigma}^G$ and for each $U \in A/G$, $x_U \in \underline{\mathcal C_U}$.

      Giving $\mathsf O_G$ and $\underline{\Sigma}^G$ the discrete simplicial enrichment,
      we define $\Sigma_G \wr \underline{\mathcal C}$ for any simplicial category $\underline{\mathcal C} \to \mathsf O_G$ over the orbit category
      to be the above pullback, taken in $\sCat$.
\end{definition}

Unpacking, the mapping spaces (or hom-sets) are given by
\begin{equation}
      \label{SGWC_MAP_EQ}
      \Map_{\Sigma_G \wr \underline{\mathcal C}}\left(
            \left(\substack{B \\ \downarrow \\ S},\ \big(y_V\big) \right),\
            \left( \substack{A \\ \downarrow \\ R},\ \big(x_U\big) \right)
      \right)
      =
      \coprod_{(q, \bar f) \in \underline{\Sigma}^G(A,B)}\Map_{\C_V}\left(
            y_V, q_V^{\**}x_{q(V)}
      \right)
\end{equation}
where $q$, $q_V$ are slight abuses of notation for the composite $q \bar f^{-1}$, resp. restricted to $V$.

\begin{example}
      Objects in $\Sigma_G \wr \Sigma_G$ are ``height 2 $G$-trees'', an example of which is displayed below.
      
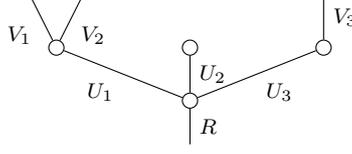
\begin{equation}
            \label{HT2TREE_EX}
            \begin{tikzpicture}
                  [grow=up,auto,level distance=2em,every node/.style = {font=\footnotesize}]
                  \tikzstyle{level 2}=[sibling distance=5em]
                  \tikzstyle{level 3}=[sibling distance=2em]
                  \node {}
                  child{node [dummy] {}
                      child{node [dummy] {}
                        child{edge from parent node [swap] {$V_3$}}
                        edge from parent node [swap] {$U_3$}
                      }
                      child{node [dummy] {}
                        edge from parent node [swap] {$U_2$}}
                      child{node [dummy] {}
                        child{edge from parent node [swap] {$V_2$}}
                        child{edge from parent node {$V_1$}}
                        edge from parent node {$U_1$}
                      }
                      edge from parent node [swap] {$R$}
                    };
              \end{tikzpicture}
        \end{equation}
\end{example}

The operation $\Sigma_G \wr (-)$ acts on the categories of Grothendieck fibrations from Definition \ref{CATFIB_DEF}.

\begin{proposition}[\cite{BPGSym}]
      The endofunctor $\Sigma_G \wr (-)$ from Definition \ref{SGW_DEF}
      is a monad on the category of fully split simplicial Grothendieck fibrations over $\mathsf O_G$. 
\end{proposition}

In particular, we have a simplicially enriched functor
\[
      \Sigma_G \wr \Sigma_G \wr \UV \xrightarrow{\sigma^0} \Sigma_G \wr \UV,
      \qquad
      \left( \substack{A \\ \downarrow \\ R},\ \left( \substack{B_U \\ \downarrow \\ U},\ \big(x_{U,V}\big)\right)\right) \longmapsto \left( \substack{\amalg B_U \\ \downarrow \\ R}, \big(x_{U,V}\big)\right).
\]

We use this monad to define our algebraic structure.

\begin{definition}
      \label{GSYM_DEF}
      A \textit{simplicial $q$-split $E\Sigma_G$-algebra} 
      is a fully split simplicial Grothendieck fibration $\UV \to \mathsf O_G$
      equipped with the structure of a pseudo-algebra over the monad $\Sigma_G \wr (-)$ in the (2,1)-category
      of fully split simplicial fibrations over $\mathsf O_G$ and maps of split fibrations. 

      Unpacking, this is the data of a simplicially-enriched functor of split fibrations over $\mathsf O_G$
      \[
            \Sigma_G \wr \UV \xrightarrow{\otimes} \UV
      \]
      and a natural simplicially enriched associativity isomorphisms
      \[
            \begin{tikzcd}
                  \Sigma_G \wr \Sigma_G \wr \UV \arrow[r, "\otimes"] \arrow[d, "\sigma^0"']
                  &
                  \Sigma_G \wr \UV \arrow[d, "\otimes"] \arrow[dl, shorten <>=10pt, Rightarrow, "\alpha"']
                  \\
                  \Sigma_G \wr \UV \arrow[r, "\otimes"]
                  &
                  \UV
            \end{tikzcd}
      \]
      which are unital and satisfy a ``pentagon identity''.

      If $\alpha$ is the identity, we say $\UV$ is \textit{fully split} or \textit{$G$-permutative}.
\end{definition}

We will often abuse notation, and omit the adjectives ``simplicial'' and ``$q$-split''.

\begin{remark}
      \label{GPERM_OTIMES_REM}
      What is written above differs from the more general definition given in \cite{BPGSym}.
      However, when restricting to the $q$-strict case, i.e. when we require that $\otimes$ is a map of split fibrations,
      the two definitions agree:
      functors $\UV^{\underline{op}} \to \underline{\mathcal W}^{\underline{op}}$ of split fibrations
      are the same data as functors $\UV \to \underline{\mathcal W}$ of split fibrations,
      even though they are not the same in general.
\end{remark}

Definition \ref{OGWRV_DEF} below provides a large class of examples: any symmetric monoidal category generates an $E \Sigma_G$-algebra.

\begin{definition}
      A \textit{strong $q$-split monoidal functor} between two $q$-split $E \Sigma_G$-algebras $\UV$ and $\underline{\mathcal W}$
      is a functor $F: \UV \to \underline{\mathcal W}$ of split fibrations over $\mathsf O_G$
      together with a natural isomorphism
      \begin{equation}
            \begin{tikzcd}
                  \Sigma_G \wr \UV \arrow[r, "F"] \arrow[d, "\otimes"']
                  &
                  \Sigma_G \wr \underline{\mathcal W} \arrow[d, "\otimes"] \arrow[dl, "\rho"', shorten <>=10pt,Rightarrow]
                  \\
                  \UV \arrow[r, "F"]
                  &
                  \underline{\mathcal W}
            \end{tikzcd}
      \end{equation}
      which is compatible with the associativity isomorphisms of $\UV$ and $\underline{\mathcal W}$.
      
      We denote the category of simplicial $q$-split (resp. fully-split) $E \Sigma_G$-algebras
      and strong $q$-split monoidal functors by $\mathsf{sSymMon}^q_G$ (resp. $\mathsf{sPerm}_G$),
      and $\mathsf{sSymMon}^q_{G,f}$ (resp. $\mathsf{sPerm}_{G,f}$) for the full subcategories spanned by the locally fibrant $\UV$.
\end{definition}

In \cite{Bon_Coh}, we establish the following coherency result using an extension of Mac Lane's construction,
and as a consequence we have that 
``any diagram of associators commutes''.
\begin{theorem}
      \label{COH_THM}
      The inclusion $\mathsf{sPerm}_{G,f} \into \mathsf{sSymMon}^q_{G,f}$ is an equivalence of categories.
\end{theorem}


We are now in position to extend the construction $\mathsf{SymMon} \into \Op(\Set)$ into the genuine equivariant setting.

\begin{remark}
      \label{FIBERHOM_REM}
      Given any simplicial split Grothendieck fibration $\UV$, we have an additional ``fiberwise'' mapping space functor 
      \begin{equation}
            \label{VOG_HOM_EQ}
            \Map_\bullet: \UV \times_{O_G} \UV^{\underline{op}} \longto \sSet^{op}
            \qquad            
            (U,x,y) \longmapsto \Map_{\V_{U}}(x,y),
      \end{equation}
      where $\V^{\underline{op}}$ is the ``fiberwise opposite'' category, i.e. the Grothendieck fibration associated to
      \begin{equation}
            O_G^{op} \longto \sCat, \qquad
            U \longmapsto \V_{U}^{op}.
      \end{equation}      
      A map $(V,\bar x,\bar y) \to (U, x, y)$ in $\UV \times_{\mathsf O_G} \UV^{\underline{op}}$
      is given by a map $q: V \to U$ in $\mathsf O_G$ and maps $f: \bar x \to q^{\**} x$, $g: q^{\**} y \to \bar y$,
      and $\Map_\bullet$ sends this triple to the composite
      \begin{equation}
            \Map_{\V_{U}}(x, y) \xrightarrow{q^{\**}}
            \Map_{\V_{V}}(q^{\**} x, q^{\**} y) \xrightarrow{f^{\**} g_{\**}}
            \Map_{\V_V}(\bar x, \bar y);
      \end{equation}
      an easy adjunction argument shows this functor is in fact enriched.
\end{remark}

\begin{proposition}
      \label{PERMG_OPG_PROP}
      There is a faithful functor $\P_{(-)}: \mathsf{sSymMon}^q_{G} \to \sOp_G$ from
      $q$-split $E \Sigma_G$-algebras
      to multicolored genuine equivariant operads. 
\end{proposition}
\begin{proof}
      Fix an $E \Sigma_G$-algebra
      $\UV$, and let $\mathfrak C$ denote the coefficient system of objects.
      Define the $\mathfrak C$-colored $G$-symmetric sequence $\P_{\UV}$ to be the (opposite of the) following composite:
      \begin{equation}
            \P_{\UV}^{op}:
            \underline{\Sigma}^G_{\mathfrak C} \simeq \Sigma_G \wr \mathfrak C \times_{O_G} \mathfrak C^{\underline{op}}
            \to \Sigma_G \wr \UV \times_{O_G} \V^{\underline{op}}
            \xrightarrow{\otimes} \UV \times_{O_G} \V^{\underline{op}}
            \xrightarrow{\Map_{\bullet}} \Set^{op}.
      \end{equation}
      Explicitly, $\P_{\UV}(A \to R, ((x_U); x_R)) := \Map_{\V_R}((x_U)^{\otimes A \to R}, x_R)$,
      and
      for arrows $f \in \underline{\Sigma}^G_{\mathfrak C}$ as in \eqref{SGC_MAP_EQ},
      define
      \begin{align*}
        \P_{\UV}(f) \colon \Map_{\UV_R}((x_U)^{\otimes (A \to R)}, x_R) & \longto \Map_{\UV_S}((q_V^{\**}x_{q(V)})^{\otimes (B \to S)}, q^{\**}x_R),
        \\
        \phi & \longmapsto
               \left(
               (q_V^{\**}x_{q(v)})^{\otimes (B \to S)} \xrightarrow{=}
               q^{\**}((x_U)^{\otimes (A \to R)}) \xrightarrow{q^{\**}(\phi)}
               q^{\**}(x_R)
               \right)
      \end{align*}
      where the first map in the image of $\phi$ is a bijection since $\otimes$ is a map of split fibrations.

      The composition and associativity of $\UV$ endow this genuine equivariant symmetric sequence with the structure of a genuine operad,
      via maps of the form
      \begin{equation}
            \resizebox{1\textwidth}{!}{$
              \Map_{\UV_R}\Big((x_U)^{\otimes (A \to R)}, x_R\Big) \times \prod_{U} \Map_{\UV_U}\Big((x_{U,V})^{\otimes (B_U \to U)}, x_U\Big)
              \to
              \Map_{\UV_R}\Big((x_{U,V})^{\otimes (\amalg B_U \to R)}, x_R\Big),
              $}
      \end{equation}
      \\[-20pt]
      \begin{equation}
            \label{PV_COMP_EQ}
            \resizebox{1\textwidth}{!}{$
              (\phi, (\psi_U))
              \mapsto
              \left(
                    (x_{U,V})^{\otimes (\amalg B_U \to R)} \xrightarrow{\alpha}
                    \left( \left(x_{U,V}\right)^{\otimes (B_U \to U)}\right)^{\otimes (A \to R)} \xrightarrow{(\psi_U)^{\otimes (A \to R)}}
                    (x_U)^{\otimes (A \to R)} \xrightarrow{\phi}
                    x_R
              \right).
              $}
      \end{equation}
      Associativity and unitality of $\P_{\UV}$ follow from the coherence of associativity and unitality of $\UV$.
      

      Now, suppose we have a strong $q$-split map $(F,\rho): (\UV, \otimes) \to (\underline{\mathcal W}, \otimes)$.
      Define $F: \P_{\UV} \to \P_{\underline{\mathcal W}}$ on an object $(A \to R, ((x_U); x_R)) \in \underline{\Sigma}^G_{\mathfrak C}$ by
      \begin{equation}
            \resizebox{1\textwidth}{!}{$
              \Map_{\UV_R}\Big((x_U)^{\otimes (A \to R)}, x_R\Big) \xrightarrow{F}
              \Map_{\underline{\mathcal W}_R}\Big( F\big((x_U)^{\otimes (A \to R)}\big), F(x_R) \Big) \xrightarrow{\phi^{\**}}
              \Map_{\underline{\mathcal W}_R}\left( \Big(F(x_U)\Big)^{\otimes (A \to R)}, F(x_R) \right).
              $}
      \end{equation}
      A simple diagram chase, using the fact that $\UV$, $\underline{\mathcal W}$, and $F$ are all $q$-split,
      shows that this map is natural in $(A \to R, ((x_U); x_R))$.
      Moreover, $F$ is a map of genuine equivariant operads:
      For any compatible collections as in \eqref{PV_COMP_EQ}, we have the diagram bellow.
      \begin{equation}
            \begin{tikzpicture}[baseline= (a).base]
                  \node[scale=.64] (a) at (0,0){
                    \begin{tikzcd}[column sep = small]
                          \left( F(x_{U,V}) \right)^{\otimes (\amalg B_U \to R)} \arrow[d, "\rho"'] \arrow[r, "\alpha"]
                          &
                          \left( (F(x_{U,V}))^{\otimes (B_U \to U)} \right)^{\otimes (A \to R)} \arrow[r, "\rho"] 
                          &
                          \left( F \left((x_{U,V})^{\otimes (B_U \to U)}\right)\right)^{\otimes (A \to R)} \arrow[r, "\psi_U"] \arrow[d, "\rho"] 
                          &
                          \left( F(x_U) \right)^{\otimes (A \to R)} \arrow[d, "\rho"] \arrow[r, "\rho"]
                          &
                          F\left((x_U)^{\otimes (A \to R)} \right) \arrow[r, "\phi"]
                          &
                          F(x_R) \arrow[d, equal]
                          \\
                          F\left((x_{U,V})^{\otimes (\amalg B_U \to R)}\right) \arrow[rr, "\alpha"]
                          &&
                          F \left(\left((x_{U,V})^{\otimes (B_U \to U)}\right)^{\otimes (A \to R)}\right) \arrow[r, "\psi_U"] 
                          &
                          F \left((x_U)^{\otimes (A \to R)} \right) \arrow[rr, "\phi"]
                          &&
                          F(x_R)
                    \end{tikzcd}
                  };
            \end{tikzpicture}
      \end{equation}      
      The left square is precisely the compatibility condition for $\rho$ and hence commutes,
      while the middle square commutes by the naturality of $\rho$.

      Finally, this functor is faithful, as the original map $F$ can be recovered from $F: \P_{\UV} \to \P_{\underline{\mathcal W}}$
      by its actions on the objects $(U \to U, ((x_U); y_U))$.
\end{proof}

We record a result of this proof.
\begin{lemma}
      \label{QKSI_ID_LEM}
      Fix an $E \Sigma_G$-algebra $\UV$.
      Let $\vect{C} = (A \to R, ((x_U); (x_U)^{\otimes (A \to R)})) \in \underline{\Sigma}^G_{\mathfrak C_{\UV}}$
      be a $\mathfrak C(\UV)$-signature, 
      and
      \[
            \ksi \in \P_{\UV}(\vect{C})
            = \Map_{\UV_R}\big((x_U)^{\otimes (A \to R)}, (x_U)^{\otimes (A \to R)}\big)
      \]
      the identity.
      Then for all $q: S \to R$ in $\mathsf O_G$,
      $q^{\**}: \P_{\UV}(\underline C) \to \P_{\UV}(q^{\**}\underline C)$
      sends $\ksi$ to the identity.
\end{lemma}

\begin{remark}
      We note that the definition given of $\mathcal P_{\UV}$ is not well-defined unless $(\UV, \otimes)$ is $q$-split,
      and $F: \P_{\UV} \to \P_{\underline{\mathcal W}}$ is not even natural in $(A \to R) \in \underline{\Sigma}^G$ unless $F$ itself was $q$-split.
\end{remark}

\subsection{Proof of Theorem \ref{THMIII}} 
\label{PROOFIII_SEC}

In this section, we characterize the image of $\P_{(-)}$ in terms of operadic op-fibrations, and build an inverse functor.
Specifically, we prove the following technical version of Theorem \ref{THMIII}.

\begin{theorem}
      \label{THMIII_PRECISE}
      The faithful inclusion of categories
      \[
            \mathsf{sSymMon}^q_{G} \into \sOp_G,
      \]
      from Proposition \ref{PERMG_OPG_PROP}
      restricts to compatible isomorphisms of categories
      \begin{equation}
            \begin{tikzcd}
                  \mathsf{sPerm}_{G,f} \arrow[r, "\P_{(-)}", "\cong"'] \arrow[d, hookrightarrow, "\simeq"']
                  &
                  \Fib^f(\mathsf{Comm}) \arrow[d, hookrightarrow, "\simeq"] 
                  \arrow[r, hookrightarrow]
                  &
                  \sOp_{G,f} \arrow[d, equal]
                  \\
                  \mathsf{sSymMon}^q_{G,f}  \arrow[r, "\P_{(-)}"', "\cong"]
                  &
                  \Fib^q(\mathsf{Comm})  
                  \arrow[r, hookrightarrow]
                  &
                  \sOp_{G,f},
            \end{tikzcd}
      \end{equation}
      where $\Fib^f(\mathsf{Comm}) \subseteq \Fib^q(\mathsf{Comm}) \subseteq \sOp_G$ are defined as in Definition \ref{FIBQCAT_DEF}.
\end{theorem}
This is accomplished in three mains steps: Proposition \ref{SYMFIB_PROP} establishes that the map $\P_{(-)}$ restricts as above,
Proposition \ref{VPSPLIT_PROP} proves that we have a well-defined map in the opposite direction,
and the proof of Theorem \ref{THMIII_PRECISE} follows from identifying the compatibility between the two constructions.

First, we identify the cocartesian arrows in $\P_{\UV}$.
\begin{lemma}
      \label{ISOCOCART_LEM}
      For all $(A \to R) \in \underline{\Sigma}^G$ and tuples of objects $(x_U)_{U \in A/G}$ with $x_U \in \UV_U$,
      an operation $\ksi \in \P_{\UV}(A \to R, ((x_U); x_R))$ is cocartesian iff
      $\ksi: (x_U)^{\otimes (A \to R)} \to x_R$ is an isomorphism in $\V_R$.
\end{lemma}
\begin{proof}
      This follows immediately from the composition structure of $\P_{\UV}$ from \eqref{PV_COMP_EQ}.
\end{proof}

\begin{proposition}
      \label{SYMFIB_PROP}
      For any $q$-split (resp. fully split) $E \Sigma_G$-algebra $\UV$,
      $\mathcal P_{\UV}$ is a $q$-split (resp. fully split) op-fibrant genuine equivariant operad.
\end{proposition}
\begin{proof}
      The identity map
      \[
            id = \ksi_{(x_U)} \in \mathcal P_{\UV}\left( (x_U), (x_U)^{\otimes (A \to R)}\right)
            = \Map_{\UV_R}\left( (x_U)^{\otimes (A \to R)}, (x_U)^{\otimes (A \to R)} \right)
      \]
      is a cocartesian lift by Lemma \ref{ISOCOCART_LEM}.
      Moreover, Lemma \ref{QKSI_ID_LEM} then implies that these choices are natural in $(A \to R) \in \underline{\Sigma}^G$.
      Finally, we observe that the composite of chosen cocartesian arrows
      is an instance of the natural isomorphism $\alpha$,
      and thus these composites are all the identity iff $\alpha$ is the identity.
\end{proof}

We will now show that these split op-fibrant genuine equivariant operads are precisely the image of $\mathsf{sSymMon}^q_{G,f}$
by defining an inverse operation.

\begin{definition}
      Fix a $q$-split op-fibrant object $\P \in \sOp_{G,f}$ with coefficient system of colors $\mathfrak C$.
      Define the coefficient system $\UV = \UV[\P]$ by
      setting $\UV_U$ to be the simplicial category with object set $\mathfrak C_U$ and mapping spaces
      \[
            \Map_{\UV_U}(x,y) = \P \left( \varrow{U}{U}, (x; y) \right).
      \]
      Given $q: V \to U$ in $\mathsf O_G$, define the restriction map $\UV_U \to \UV_V$ by
      \[
            x \longmapsto q^{\**}x, \qquad
            \P \left( \varrow{U}{U}, (x; y) \right) \xrightarrow{q^{\**}} \P \left( \varrow{V}{V}, (q^{\**}x; q^{\**}y) \right).
      \]
      
      Given $(A \to R) \in \underline{\Sigma}^G$ and objects $x_U \in \mathfrak C_U$ for each $U \in A/G$,
      let $(x_U)^{\otimes (A \to R)}$ denote the codomain of the chosen cocartesian arrow associated to $(A \to R, (x_U))$,
      and denote the arrow itself by
      \[
            \ksi_{(x_U)}: (x_U) \longto (x_U)^{\otimes (A \to R)},
            \qquad
            \ksi_{(x_U)} \in \P\left( \substack{A \\ \downarrow \\ R},\ \Big((x_U); (x_U)^{\otimes (A \to R)}\Big)\right).
      \]
      We define the genuine monoidal product
      $\Sigma_G \wr \UV \xrightarrow{\otimes} \UV$
      on objects by
      $(A \to R, (x_U)) \mapsto (x_U)^{\otimes (A \to R)}$.
      Given an arrow\footnote{
        We warn that this is now the reverse of what we saw in \eqref{SGC_MAP_EQ}.},
      $(q, (f_V)): (B \to S, (y_V)) \to (A \to R, (x_U))$ in $\underline{\Sigma}^{G,op}_{\mathfrak C}$
      define the associated arrow in $\UV$ by
      \[
            (y_V)^{\otimes (B \to S)} \xrightarrow{(f_V)^{\otimes (B \to S)}}
            (q_V^{\**}x_{q(V)}) = q^{\**}\left((x_U)^{\otimes (A \to R)}\right) \xrightarrow{q}
            (x_U)^{\otimes (A \to R)},
      \]
      where $(f_V)^{\otimes (B \to S)}$ is the unique operation (via Lemma \ref{GOH_COCART_LEM}(iii)) such that the following commutes
      \begin{equation}
            \label{OT_OF_MAPS_EQ}
            \begin{tikzcd}
                  (y_V) \arrow[r, "{(f_V)}"] \arrow[d, "\ksi_{(y_V)}"]
                  &
                  (q_V^{\**}x_{q(V)}) \arrow[r, "\ksi_{q^{\**}(x_U)}"]
                  &
                  (q_V^{\**}x_{q(V)})^{\otimes D}
                  \\
                  (y_V)^{\otimes D}, \arrow[urr, dashed, "\exists!\ (f_v)^{\otimes D}"']
            \end{tikzcd}
      \end{equation}
      and we know $(q_V^{\**}x_{q(V)}) = q^{\**}((x_U)^{\otimes C})$ since $\P$ is $q$-split.
\end{definition}

\begin{lemma}
      The above multiplication map $\otimes: \Sigma_G \wr \UV[\P] \to \UV[\P]$ is functorial, and moreover a map of split fibrations over $\mathsf O_G$.
\end{lemma}
\begin{proof}
      The ``moreover'' statement follows by the naturality of the chosen cocartesian arrows.
      Given composable maps
      \[
            (E, (z_W)) \xrightarrow{(p, (g_W))} (D, (y_V)) \xrightarrow{(q, (f_V))} (C, (x_U))
      \]
      in $\Sigma_G \wr \UV[\P]$,
      the first claim holds since the following diagram commutes,
      \begin{equation}
            \begin{tikzcd}
                  (z_W) \arrow[r, "{(g_W)}"] \arrow[d, "\ksi_{(z_W)}"']
                  &
                  (y_W) = p^{\**}(y_V) \arrow[r, "{p^{\**}\left( (f_V) \right)}"] \arrow[d, "\ksi_{(y_W)}"', "p^{\**}\ksi_{(y_V)}"]
                  &[10pt]
                  p^{\**}\left( (x_V) \right) = (x_W) \arrow[d, "p^{\**} \ksi_{(x_V)}"', "\ksi_{(x_W)}"]
                  \\
                  (z_W)^{\otimes E} \arrow[r, "{(g_W)^{\otimes E}}"']
                  &
                  (y_W)^{\otimes E} = p^{\**}\left( (y_V)^{\otimes D} \right) \arrow[r, "{p^{\**}\left((f_V)^{\otimes D} \right)}"']
                  &
                  p^{\**}\left((x_V)^{\otimes D}\right) = (x_W)^{\otimes E},
            \end{tikzcd}
      \end{equation}
      where
      \[
            x_V = q_V^{\**}x_{q(V)},
            \quad
            x_W = p_W^{\**}x_{p(W)},
            \quad
            y_W = p_W^{\**}y_{p(W)}.
      \]
\end{proof}

\begin{proposition}
      \label{VPSPLIT_PROP}
      For $\P \in \sOp_{G,f}$ $q$-split (resp. fully split) op-fibrant,
      $(\UV[\P], \otimes)$ is a $q$-split (resp. fully split) $E \Sigma_G$-algebra
\end{proposition}
\begin{proof}
      %
      Let $\mathfrak C$ be the coefficient system of colors associated to $\P$.
      Consider an element $(A \to R, (B_U \to U, (x_{U,V}))) \in \Sigma_G \wr \Sigma_G \wr \UV$,
      so $x_{U,V} \in \mathfrak C_V$ for all $V \in B_U/G$ and all $U \in \mathbf A/G$,
      and let $B = \amalg B_U$.
      We will build a natural isomorphisms
      \mbox{$(x_V)^{\otimes (B \to R)} \to \left( (x_V)^{\otimes (B_U \to U)} \right)^{\otimes (A \to R)}$}
      and it's inverse.
      Let 
      \begin{align*}
        \ksi_B: & (x_V)_{V \in \mathbf B/G} \longto (x_V)^{\otimes (B \to R)},
        \\
        \ksi_U: & (x_V)_{V \in \mathbf B_U/G} \longto (x_V)^{\otimes (B_U \to U)},
        \\
        \ksi_A: & \left((x_V)^{\otimes (B_U \to U)}\right)_{U \in \mathbf A/G} \longto \left((x_V)^{\otimes (B_U \to U)}\right)^{\otimes (A \to R)}
      \end{align*}
      denote the chosen cocartesian arrows.
      
      First, define $\alpha: (x_V)^{\otimes (B \to R)} \to \left( (x_V)^{\otimes (B_U \to U)} \right)^{\otimes (A \to R)}$
      in $\P(A \to R)$ to be the unique operation (since $\ksi_B$ is cocartesian) such that
      $\alpha \circ \ksi_B = \ksi_A \circ (\ksi_U)$.
      \begin{equation}
            \label{KSICOMP_EQ}
            \begin{tikzcd}
                  (x_V) \arrow[r, "{(\ksi_U)}"] \arrow[d, "\ksi_B"]
                  &
                  \left((x_V)^{\otimes (B_U \to U)}\right) \arrow[r, "\ksi_A"]
                  &
                  \left((x_V)^{\otimes (B_U \to U)}\right)^{\otimes (A \to R)}
                  \\
                  (x_V)^{\otimes (B \to R)} \arrow[urr, dashed, "\exists ! \alpha"'].                  
            \end{tikzcd}
      \end{equation}

      Conversely, by Lemma \ref{GOH_COCART_LEM}(iii) there exists a unique $\beta: ((x_V)^{\otimes (B_U \to U)}) \to (x_V)^{\otimes (B \to R)}$ in $\P(A \to R)$ such that
      \mbox{$\beta \circ (\ksi_U) = \ksi_{B}$}.
      Then
      \[
            \alpha \circ \beta \circ (\ksi_U) = \alpha \circ \ksi_B = \ksi_A \circ (\ksi_U),
      \]
      and hence Lemma \ref{GOH_COCART_LEM})(iii) implies $\alpha \circ \beta = \ksi_A$.

      Now, let $\gamma: \left((x_V)^{\otimes (B_U \to U)}\right)^{\otimes (A \to R)} \to (x_V)^{\otimes (B \to R)}$ in $\P(A \to R)$
      denote the unique operation such that $\gamma \circ \ksi_A = \beta$.
      We claim $\alpha$ and $\gamma$ are inverse natural isomorphisms.
      We observe that
      \[
            \gamma \circ \alpha \circ \ksi_B = \gamma \circ \ksi_A \circ (\ksi_U) = \beta \circ (\ksi_U) = \ksi_B,
      \]
      and hence uniqueness implies $\gamma \circ \alpha = id$.
      Similarly,
      \[
            \alpha \circ \gamma \circ \ksi_A \circ (\ksi_U) = \alpha \circ \beta \circ (\ksi_U) = \alpha \circ \ksi_B = \ksi_A \circ (\ksi_U),
      \]
      and again Lemma \ref{GOH_COCART_LEM}(iii) implies that $\alpha \circ \gamma = id$.
       
      Second, naturality of $\alpha$ and $\gamma$ follow from the naturality of the chosen cocartesian arrows $\ksi$.

      Third, unitality and the ``pentagon identity'' for $(\UV[\P], \otimes, \alpha)$ follow from analogous arguments as above,
      using the uniqueness of these factorizations involving the cocartesian arrows.

      Finally, if the composite of chosen cocartesian arrows is a chosen cocartesian arrow, then
      by considering \eqref{KSICOMP_EQ} we conclude that $\alpha$ must be the identity.
\end{proof}

We may now prove Theorem \ref{THMIII_PRECISE}.

\begin{proof}
      [Proof of Theorem \ref{THMIII_PRECISE}] 
      On objects, $\P_{(-)}$ and $\UV[-]$ are inverses by
      \eqref{COCARTPULL_EQ} 
      and by unpacking definitions and using the fact that our chosen lifts in $\P_{\UV}$ are the identities.

      On arrows, for any $q$-split strong monoidal $F \colon \UV \to \underline{\mathcal W}$,
      $\P_{F}$ sends cocartesian morphisms to their composite with the associated component of the natural isomorphism $\rho$,
      which is again cocartesian by Lemma \ref{ISOCOCART_LEM}.
      Conversely, any $F: \P_{\UV} \to \P_{\underline{\mathcal W}}$ induces a map of coefficient systems $\UV \to \underline{\mathcal W}$,
      and if $F$ also preserves cocartesian arrows,
      we define $\rho_{(x_U)}$ to be the image under $F$ of the chosen cocartesian (identity) maps $\ksi_{(x_U)}$.
      It is straightforward to check that this produces a strong $q$-split monoidal simplicial functor,
      and that these operations are inverse on hom-sets.
\end{proof}

\begin{notation}
      By abuse of notation, we will use $(\UV, \otimes)$ to denote either a $E \Sigma_G$-algebra or its image in $\sOp_G$.
\end{notation}

\subsection{Proof of Theorem \ref{THMII}} 
\label{PROOFII_SEC}

We will now show that the subcategory $\mathsf{sSymMon}^q_G$ in $\sOp_{G,f}$ maps under $N^{\otimes}$ to the
(1)-subcategory of $G$-symmetric monoidal $G$-$\infty$-categories $\mathsf{SymMon}_{\infty,G}$ inside the (1)-category of $\mathsf O_G$-$\infty$-operads.
The bulk of the work is in Proposition \ref{SPLIT_PROP1},
which translates fibration information in $\sOp_G$ to fibration information in $\sCat$.

First, recalling Definition \ref{OG_INF_OP_DEF}(iii) and Remark \ref{SEGALTYPE_REM}, we consider the following.

\begin{definition}[\cite{Nar17},\cite{BDGNS}]      
      A \textit{$G$-symmetric monoidal $G$-$\infty$-category} is
      an $\infty$-category $\mathcal C$ equipped with a map $F:\mathcal C \to \underline{\Fin}^G_{\**}$
      which is a cocartesian fibration in $\sSet$ of Segal type.
      A \textit{monoidal functor} between $G$-symmetric monoidal $G$-$\infty$-categories is a map of fibrations over $\underline{\Fin}^G_{\**}$, i.e. it preserves cocartesian arrows.
      We denote this (1)-category by $\mathsf{SymMon}_{\infty,G}$.
\end{definition}

Moving back to the category of (1)-categories briefly, we make the following definition.

\begin{definition}[{cf. Definition \ref{OG_INF_OP_DEF}(iii), Remark \ref{SEGALTYPE_REM}}]
      Let $\mathcal C \to \underline{\Fin}^G_{\**}$ be a (split) simplicial Grothendieck op-fibration of categories.
      We say $\mathcal C$ is of \textit{Segal type} if
      for all objects $(A \to R) \in \underline{\Fin}^G_{\**}$,
      the product of the maps induced by the (chosen) cocartesian liftings against the inert projection maps $\pi_U \in \underline{\Fin}^G_{\**}(A \to R, U \to R)$
      \begin{equation}
            \label{SEGALTYPE_EQ}
            \mathcal C_{\langle A \to R \rangle} \xrightarrow{(\pi_U)} \prod_{U \in A/G} \mathcal C_{\langle U \to R \rangle}
      \end{equation}
      is an equivalence of simplicial categories.

      Extending Definition \ref{CATFIB_DEF}, we write $\Fib^f_{\mathrm{Segal}}(\underline{\Fin}^G_{\**}) \subseteq \Fib^q_{\mathrm{Segal}}(\underline{\Fin}^G_{\**})$
      for the full subcategories of $\Fib^f(\underline{\Fin}^G_{\**}) \subseteq \Fib^q(\underline{\Fin}^G_{\**})$
      spanned by op-fibrations of Segal type.
\end{definition}

We have the following.

\begin{proposition}
      \label{SPLIT_PROP1}
      Let $F:\P \to \Comm$ be a map in $\sOp_G$. Then
      $F$ is a ($q$-split, fully split) genuine operadic op-fibration
      if and only if
      $F^\otimes: \P^{\otimes} \to \Comm^{\otimes} = \underline{\Fin}^G_{\**}$ is a
      ($q$-split, fully split) Grothendieck op-fibration of Segal type.
\end{proposition}
\begin{proof}
      Suppose $F^\otimes$ is a ($q$-split, fully split) genuine operadic op-fibration.
      Fix $(A \to R)$, $(B \to S)$ in $\underline{\Fin}^G_{\**}$, an arrow $f = (q, \bar f) \in \underline{\Fin}^G_{\**}(A,B)$, 
      and an object $(A \to R, (x_U)) \in \P^{\otimes}$ over $(A \to R)$.
      For each $V \in B/G$, let
      \[
            \ksi_V: (x_{\bar U)})_{\bar U \in \bar f^{-1}(V)/G} \longto x_V,
            \qquad
            x_{\bar U} = q_{\bar U}^{\**}x_{q(\bar U)},
      \]    
      denote a (chosen) cocartesian lift in $\P$ with arity $(f^{-1}(V) \to V)$ and the given source.
      Then the collection $(\ksi_V)$ is in fact a lift of $f$ in $\P^{\otimes}$ with source $(A \to R, (x_U))$.

      Moreover, we claim it is $F^{\otimes}$-cocartesian.
      Given another object $(C \to T)$, we can identify the composition map $(\ksi_V)^{\**}$ as in the diagram below.
      \[
            \begin{tikzcd}[row sep = small]
                  \P^{\otimes}\left(
                        \left(\substack{B \\ \downarrow \\ S},\ (x_V)\right),\ \left(\substack{C \\ \downarrow \\ T},\ (z_W)\right)
                  \right) \arrow[r, "{(\ksi_V)^{\**}}"] \arrow[d, equal]
                  &
                  \P^{\otimes}\left(
                        \left(\substack{A \\ \downarrow \\ R},\ (x_U)\right),\ \left(\substack{C \\ \downarrow \\ T},\ (z_W)\right)
                  \right) \arrow[d, equal]
                  \\
                  \displaystyle{
                    \coprod_{(p,\bar g)} \prod_{W \in C/G} \P \left(
                          \substack{\bar g^{-1}(W) \\ \downarrow \\ W},\ \Big( \big( p_{\bar V}^{\**} x_{p(\bar V)} \big); z_W \Big)
                    \right) }
                  \arrow[r, "{\prod p_{\bar V}^{\**} \ksi_{p(\bar V)}}"]
                  &
                  \displaystyle{
                    \coprod_{(p,\bar g)} \prod_{W \in C/G} \P \left(
                          \substack{\bar {gf}^{-1}(W) \\ \downarrow \\ W},\ \big( x_{\ddot{U}} \big); z_W \Big) 
                    \right) }
            \end{tikzcd}
      \]
      where for each $\ddot{U} \in (qp)^{\**}A /G$, we define the color $x_{\ddot{U}}$ in $\mathfrak C_{\ddot{U}} = \mathfrak C(\P)_{\ddot{U}}$ to be
      the image of $x_{qp(\ddot{U})}$ under either map below (cf. \eqref{TOOMANYEQ_EQ}).
      \[
            \begin{tikzcd}[column sep = small]
                  \mathfrak C_{qp(\ddot{U})} \arrow[rr, "{q_{p(\ddot{U})}^{\**}}"] \arrow[dr, "{(qp)_{\ddot{U}}^{\**}}"']
                  &&
                  \mathfrak C_{p*(\ddot{U})} \arrow[dl, "{p_{\ddot{U}}^{\**}}"]
                  \\
                  &
                  \mathfrak C_{\ddot{U}}
            \end{tikzcd}
      \]
      These maps are all well-defined by \eqref{PULLASSEM_EQ},
      each $p_{\bar V}^{\**}\ksi_{p(\bar V)}$ is cocartesian since $\ksi_V$ is cocartesian,
      and thus the map is an isomorphism by Lemma \ref{GOH_COCART_LEM}(iii).

      Conversely, if $F^{\otimes}$ is a simplicial Grothendieck op-fibration, then
      the (chosen) cocartesian arrow of $\P^{\otimes}$ over the canonical map $(A \to R) \to (R \to R)$ with source $(A \to R, (x_U))$
      is precisely an operation, cocartesian by Lemma \ref{PRECOCART_LEM}, for $\P$ with source $(x_U)$ and arity $(A \to R)$.
      
      Lastly,
      naturality and composite stability of chosen lifts in fully split op-fibrant $\P \in \Op_G$
      exactly correspond to naturality and composite stability of chosen lifts in $\P^{\otimes}$.
\end{proof}

\begin{proposition}
      \label{SPLIT_PROP2}
      Suppose $p: \mathcal C^{\otimes} \to \underline{\Fin}_{G,\**}$ is a Grothendieck op-fibration of Segal type.
      Then the homotopy coherent nerve $N(p)$ of $p$ in $\sSet$ is a cocartesian fibration of Segal type.
\end{proposition}
\begin{proof}
      Since Grothendieck op-fibrations are in particular local fibrations, $N(p)$ is a cocartesian fibration by \cite[Lemma 2.4.1.10(ii)]{Lur09}).
      As $N$ is right adjoint, it preserves pullbacks, products, and equivalences,
      and hence translates one Segal type condition to the other.
\end{proof}

Putting these pieces together, we have our proof.

\begin{proof}
      [Proof of Theorem \ref{THMII}] 
      Since the category $\mathsf{sSymMon}^q_{G,f}$ is isomorphic to the category of $q$-split op-fibrant objects in $\Op_G$,
      the composite
      \[
            \mathsf{sSymMon}^q_{G,f} \into \sOp_{G,f} \xrightarrow{N^{\otimes}} \Op_{\infty,G}
      \]
      factors through $\mathsf{SymMon}_{\infty,G}$
      by combining Propositions \ref{SPLIT_PROP1} and \ref{SPLIT_PROP2}.
\end{proof}

We end this section by showing that the above functor also preserves the underlying categories.
\begin{definition}
      Given $\UV \in \mathsf{sSymMon}^q_{G,f}$, the \textit{underlying $\mathsf O_G$-category} is the underlying Grothendieck fibration $\UV \to \mathsf O_G$.

      Given a $G$-symmetric monoidal $G$-$\infty$-category $\UV^{\otimes}$, the \textit{underlying $\mathsf O_G$-$\infty$-category}
      is the cocartesian fibration given by the left pullback square below,
      while the \textit{underlying symmetric monoidal $\infty$-category} is given by the right pullback square.
      \[
            \begin{tikzcd}
                  \UV \arrow[d] \arrow[r]
                  &
                  \UV^{\otimes} \arrow[d]
                  &
                  \V^{\otimes} \arrow[l] \arrow[d]
                  \\
                  \mathsf O_G^{op} \arrow[r]
                  &
                  \underline{\Fin}^G_{\**}
                  &
                  \Fin_{\**} \arrow[l]                  
            \end{tikzcd}
            \]
\end{definition}

Unpacking definitions, the following is clear.
\begin{lemma}
      \label{DUALGR_LEM}
      If $\mathcal C \to \mathcal B$ is a fully split Grothendieck fibration, then
      $\mathcal C^{op,\underline{op}} \to \mathcal B^{op}$ is the associated dual fully split Grothendieck op-fibration
      (cf. Remark \ref{FIBERHOM_REM}).
\end{lemma}

\begin{corollary}
      \label{UNDERLY_COR}
      Fix $(\UV, \otimes) \in \mathsf{sSymMon}^q_{G,f}$.
      Then $N(\UV^{op,\underline{op}})$ is an $\mathsf O_G$-$\infty$-category.
      
      Moreover, the underlying $\mathsf O_G$-$\infty$-category associated to the $G$-symmetric monoidal $\mathsf O_G$-$\infty$-category $N^{\otimes}(\UV, \otimes)$
      is equivalent to $N(\UV^{op,\underline{op}})$,
      and the underlying symmetric monoidal $\infty$-category is equivalent to $N^\otimes(\UV_{G/G}, \otimes)$,
      the non-equivariant operadic nerve of the symmetric monoidal simplicial category $(\UV_{G/G}, \otimes)$. 
\end{corollary}
\begin{proof}
      The first claim follows by Lemma \ref{DUALGR_LEM} and \cite[Lemma 2.4.1.10(ii)]{Lur09},
      while the moreover follows from the straightforward check that the squares below are pullbacks of simplicial categories.
      \[
            \begin{tikzcd}
                  \UV^{op,\underline{op}} \arrow[r] \arrow[d]
                  &
                  (\UV, \otimes)^{\otimes} \arrow[d]
                  &
                  (\UV_{G/G}, \otimes)^{\otimes} \arrow[l] \arrow[d]
                  \\
                  \mathsf O_G^{op} \arrow[r]
                  &
                  \underline{\Fin}^G_{\**}
                  &
                  \Fin_{\**} \arrow[l]
            \end{tikzcd}
      \]
\end{proof}

\section{Examples and Algebras}
\label{EXAMPLE_SEC}

As indicated in \cite[Cor. 4.40]{BP_geo}, the usual notion of equivariant simplicial operads form a reflexive subcategory of genuine equivariant simplicial operads.
Thus Theorem \ref{THMI} provides a means to convert our favorite $G$-operads into $\mathsf O_G$-$\infty$-operads.
In this section, we unpack this for four prominent examples of single-colored equivariant operads.

\begin{definition}[{\cite[\S 4.3]{BP_geo}}]
      \label{ISOPG_DEF}
      Given $\O \in \sOp^G$ with a single color, define $i_{\**}\O \in \sOp_G$ by
        \begin{equation}
            i_{\**}\O \left( \substack{A \\ \downarrow \\ R} \right)
            =
            \left( \prod_{r \in R}\O(A_{r}) \right)^G
            \simeq
            \O(|A_{{r_0}}|)^{\Gamma_{A_{r_0}}},
      \end{equation}
      where $A_{r}$ is the inverse image of $r \in R$,
      $r_0$ any fixed element of $R$,
      and $\Gamma_{A_{r}} = \Gamma(\alpha_r)$ the graph of the homomorphism structure map $\alpha_r: H_r \to \Sigma_{|A_{r}|}$
      encoding the $H_r$-action on $A_r$.
\end{definition}

\begin{definition}
      \label{GRAPH_DEF}
      We recall that a subgroup $\Gamma \leq G \times \Sigma_n$ is called a \textit{graph subgroup} if $\Gamma \cap \Sigma_n = \set{e}$.
      This is equivalent to the condition that $\Gamma$ is the graph of some homomorphism $G \geq H \to \Sigma_n$.
      
      A simplicial $G$-operad $\O \in \sOp^G$ with a single color is called \textit{$G$-graph fibrant} if
      for all $n \geq 0$ and all graph subgroups $\Gamma \leq G \times \Sigma_n$,
      $\O(n)^\Gamma$ is a fibrant simplicial set.
\end{definition}

The main result of \cite{BP_geo} states that the inclusion $i_{\**}: \sOp^G \to \sOp_G$ is a Quillen equivalence between
the $G$-graph model structure on $\sOp^G$, where weak equivalences and fibrations are detected on graph-subgroup fixed points,
and the projective model structure on $\sOp_G$.

\begin{corollary}
      Suppose $\O \in \sOp^G$ is a $G$-graph-fibrant simplicial operad with a single color.
      Then $i_{\**}\O \in \sOp_G$ is locally fibrant, and thus there exists an associated $O_G$-$\infty$-operad $N^\otimes(\O)$.
\end{corollary}

Moreover, composition in $\O^\otimes = (i_{\**}\O)^\otimes$ is defined just as in the non-equivariant category of operators:
the functor $i_{\**}$ can be thought of as an encapsulation of the fact that
composition in $\O$ is well-defined when restricted to the fixed-point subspaces of this form.

\begin{example}
      Following Example \ref{COMM_EX}, for $\O = \mathsf{Comm} \in \sOp^G$,
      the associated $\mathsf O_G$-$\infty$-operad is simply the identity on $N(\underline{\Fin}^G_{\**})$.
\end{example}

\begin{example}
      Let $V$ be a finite-dimensional real orthogonal $G$-representation.
      The \textit{little $V$-disks} operad $\mathcal D_V$ has $n$-ary operations the space of affine embeddings
      $\Emb^{\mathrm{Aff}}(\underline{n} \times D(V), D(V))$.

      Let $T$ be an (ordered) $H$-set with $n$-elements, and $\alpha: H \to \Sigma_n$ the associated structure map.
      Then
      \begin{equation}
            \mathcal D_V(n)^{\Gamma_T} = \mathcal D_V(n)^{\Gamma(\alpha)} = \mathcal D_V(T)^H = \Emb^{\mathrm{Aff},H}(T \times D(V), D(V)),
      \end{equation}
      and moreover this space
      is homotopy equivalent to the space of $H$-equivariant embeddings $\Emb^H(T,V)$
      (for more discussion, see e.g. \cite[Lemma 1.2]{GM17}, \cite[Thm. 4.19]{BH15}).

      Thus, for any two objects $(A \to G/H)$ and $(B \to G/K)$ in $\underline{\Fin}^G_{\**}$, we see that
      \begin{equation}
            \resizebox{1 \textwidth}{!}
            {
              $\displaystyle{
                \mathcal D_V^\otimes\left (\substack{A \\ \downarrow \\ G/H}, \substack{B \\ \downarrow \\ G/K} \right)
                \simeq
                \coprod_{f \colon A \to B} \prod_{G b \in B/G}\Emb^{\mathrm{Aff},G_b}(\bar{f}^{-1}(b) \times D(V), D(V))
                \sim
                \coprod_{f \colon A \to B} \prod_{G b \in B/G}\Emb^{G_b}(\bar{f}^{-1}(b), V).
              }$
            }
      \end{equation}
      
      Now, we say that a map $\bar{f}: A_f \to B$ of $G$-spaces is \textit{$\mathcal D_V$-admissible} if for all $b \in B$, $\bar{f}^{-1}(b)$ has a $\Stab_G(b)$-equivariant embedding into $V$.
      Given an arrow $f = (q, A_f, \bar{f})$, we note that
      if $\bar{f}$ is not $\mathcal D_V$-admissible,
      then the $f$-component of $\mathcal D_V^\otimes(A,B)$ is empty.
\end{example}

Asaf Horev has constructed a completely $\mathsf O_G$-$\infty$-categorical model for the
framed little $V$-disks operad,
and has shown it is equivalent to $N^\otimes(\mathcal D_V)$ as an $\mathsf O_G$-$\infty$-operad \cite[\S 3.9]{Hor},
with applications to genuine equivariant factorization homology.
Additional uses of the $N^{\otimes}$ construction will appear in upcoming work of Horev, Inbar Klang, and Foling Zou.

\begin{example}
      [{cf. \cite[Defn 1.2]{GM17}}]
      Let $V$ be a finite-dimensional real orthogonal $G$-representation.
      Let $\Emb_V(n)$ denote the $G$-space of embeddings $\Emb(\underline n \times V, V)$.
      With the obvious composition, these assemble into the \textit{$V$-embeddings operad} $\Emb_V$.

      Now, let $R_V \subseteq E_V(1)$ denote the subspace of distance reducing embeddings.
      A \textit{Steiner path} is a map $h: I \to R_V$ with $h(1) = id$; let $P_V$ denote the $G$-space of Steiner paths.
      There is a natural ``evaluation at 0'' map $\epsilon_0: P_V \to R_V$.
      Let $\mathcal K_V(n)$ denote the $G$-space of ordered $n$-tuples of Stein paths $(h_i)$ such that $\epsilon_0(h_i)$ are all distinct.
      With composition defined by amalgamation of paths pointwise, these form the \textit{$V$-Steiner operad} $\mathcal K_V$.

      We observe that for all $H$-sets $T$ with $n$-objects and associated structure map $\alpha: H \to \Sigma_n$,
      \begin{equation}
            \mathcal K_V(n)^{\Gamma_T} = \mathcal K_V(n)^{\Gamma(\alpha)} = \mathcal K_V(T)^H
      \end{equation}
      is equal to the set of ``$H$-stable $T$-tuples of $\epsilon_0$-distinct Steiner paths'';
      that is, $T$-indexed tuples of Steiner paths $(h_t)$ with distinct $\epsilon_0$-values and $g. h_{g^{-1}t} = h_t$ for all $g \in H$ and $t \in T$.

      Additionally, by \cite[Lemma 1.5]{GM17}, we have a $G$-graph equivalence of operads $\mathcal D_V \to \mathcal K_V$,
      and so $\mathcal K_V(n)^{\Gamma_T} \sim \Emb^H(T, V)$.
\end{example}

\begin{example}
      For any $G$-set $A$, let $EA$ denote the associated \textit{chaotic} $G$-category,
      with object $G$-set $A$ and a unique morphism between any two objects.
      Now let $\mathcal P_G$ denote the equivariant Barratt-Eccles operad from \cite{GM17,GMM17}, with $\mathcal P_G(n) = \Cat(EG, E\Sigma_n) = E\Set(G,\Sigma_n)$.
      Then for any graph subgroup $\Gamma \leq G \times \Sigma_n$,
      $\mathcal P_G(n)^\Gamma \simeq E\left(\Set(G,\Sigma_n)^\Gamma\right)$, and so
      \[
            (\mathcal P_G)^{\otimes}\left( \substack{A \\ \downarrow \\ R},\ \substack{B \\ \downarrow \\ S} \right)
            \simeq
            \coprod_{(q,\bar f)} \prod_{V \in B/G} E \left( \Set(G, \Sigma_{|\bar f^{-1}(v_0)|})^{\Gamma_{\bar f^{-1}(v_0)}} \right).
      \]      
\end{example}

\subsection{$G$-symmetric monoidal $G$-$\infty$-category of strict $G$-objects}
\label{GSTRICT_SEC}

We investigate the effects of these constructions on a fundamental class of $E \Sigma_G$-algebras.

\begin{definition}[cf. \cite{BPGSym}]
      \label{OGWRV_DEF}
      Let $(\V, \square)$ be an (unbiased) symmetric monoidal simplicial category, and let $\mathsf O_G \wr \V \to \mathsf O_G$ denote the
      simplicial Grothendieck fibration associated to the functor
      \[
            \mathsf O_G^{op} \longto \sCat,
            \qquad
            U \longmapsto \V^{G \ltimes U},
      \]
      where $G \ltimes U$ denotes the action groupoid of $G$ on $U$.
      This is naturally a simplicial $E \Sigma_G$-algebra, denoted $\mathsf O_G \wr \V^{\, \square}$,
      via the composition
      \begin{equation}
            \label{SGWOG_EQ}
            \Sigma_G \wr (\mathsf O_G \wr \V) \xrightarrow{\ \simeq \ } \mathsf O_G \wr (\Sigma \wr \V) \xrightarrow{\mathsf O_G \wr \, \square} \mathsf O_G \wr \V.
      \end{equation}
\end{definition}
Explicitly, an object in the source is equivalent to the data
\[
      \left( \substack{A \\ \downarrow \\ R}, G \ltimes A \xrightarrow{X} \V \right),
\]
and the composite \eqref{SGWOG_EQ} is given on objects by
\[
      (A \to R, X)^{\otimes (A \to R)} = C^{\square}_{\**} X, \qquad
      C^{\square}_{\**}X(r) = \bigotimes_{A_r}x_a,
\]
where $C \colon A \to R$ is as given, and 
$C^{\square}_{\**}$ is the indexed monoidal product of \cite[\S A.3.2]{HHR16}.
On mapping spaces, this is given by
\begin{equation}
      \label{SGOGV_MAP_EQ}
      \begin{tikzcd}[row sep = small]
            \displaystyle{
              \Map_{\Sigma_G \wr \mathsf O_G \wr \V}\left(
                    \left(\substack{B \\ \downarrow \\ S},\ Y \right),\ \left( \substack{A \\ \downarrow \\ R},\ X \right)
              \right)}
            \arrow[d, equal] \arrow[r, "\otimes"]
            &
            \displaystyle{
              \Map_{\mathsf O_G \wr \V} \left( \big(S, D^{\otimes}_{\**}Y\big), \big(R, C^{\square}_{\**}X\big) \right)
            }
            \arrow[d, equal]
            \\
            \displaystyle{
              \coprod_{(q, \bar f)} \Map_{\V^{G \ltimes B}}\left( Y, q^{\**}X \right)
            }
            \arrow[r, "\coprod D^{\square}_{\**}"]
            &
            \displaystyle{
              \coprod_{q} \Map_{\V^{G \ltimes R}} \left( D^{\square}_{\**}Y, D^{\square}_{\**}(q^{\**}X) \right),
            }                  
      \end{tikzcd}
\end{equation}
where $C\colon A \to R$ and $D\colon B \to S$ are as given,
the equality is given by adapting \eqref{SGWC_MAP_EQ} to our case of $\underline{\mathcal C} = \mathsf O_G \wr \V$,
and $D^{\square}_{\**}(q^{\**}X)$ is naturally isomorphic to $q^{\**}(C^{\square}_{\**}X)$ by \cite[Prop. A.31]{HHR16}.

Now,
the associated genuine category of operators $(\mathsf O_G \wr \V^{\, \square}, \otimes)^{\otimes}$ 
has objects $(A \to R, G \ltimes A \xrightarrow{X} \V)$,
and, following \eqref{POTIMES_MAP_EQ} and \eqref{SGOGV_MAP_EQ}, mapping spaces of the form
\[
      \Map\left(
            \left(\substack{A \\ \downarrow \\ R},\ G \ltimes A \xrightarrow{X} \V\right),\
            \left(\substack{B \\ \downarrow \\ S},\ G \ltimes B \xrightarrow{Y} \V\right)
      \right)
      =
      \coprod_{(q,\bar f) \in \underline{\Fin}^G_{\**}(A,B)} \prod_{V \in B/G}
      \Map_{\V^{G \ltimes B}} \left(
            \left(\bar f_V\right)^{\square}_{\**} f^{\**} X, Y
      \right)
\]
where
\[
      f^{\**}X\colon G \ltimes \bar f^{-1}(B) \to G \ltimes q^{\**}A \xrightarrow{q} G \ltimes A \xrightarrow{X} \V,
      \qquad
      \bar f_V: \bar f^{-1}(B) \to V,
\]
and we are using that the following square commutes up to natural isomorphism for all covering categories $p: I \to J$ and $q: I' \to J'$.
\[
      \begin{tikzcd}
            \V^I \times \V^J \arrow[d, "p^{\otimes}_{\**} \times q^{\otimes}_{\**}"'] \arrow[r, "\simeq"]
            &
            \V^{I \amalg J} \arrow[d, "{(p \amalg q)^{\otimes}_{\**}}"]
            \\
            \V^{I'} \times \V^{J'} \arrow[r, "\simeq"]
            &
            \V^{I' \amalg J'}
      \end{tikzcd}
\]

To ensure that all of our mapping spaces are in fact Kan complexes,
we need an additional assumption on $\V$.
\begin{definition}
      \label{GLOBAL_DEF}
      We say that a symmetric monoidal simplicial category is \textit{globally fibrant} if
      the simplicial category of strict $G$-objects $\V^G$ is locally fibrant for every finite group $G$.
\end{definition}

We note that any symmetric monoidal \textit{topological} category is globally fibrant.

\begin{definition}
      \label{VINFTYG_DEF}
      Let $(\V, \square)$ be a globally fibrant symmetric monoidal simplicial category. 
      We define the \textit{$G$-symmetric monoidal $G$-$\infty$-category of strict $G$-objects in $(\V,\square)$},
      denoted $\UV_{\infty,G}^{\, \square}$, to be $N^\otimes(\mathsf O_G \wr \V^{\, \square}, \otimes)$ 
      the genuine operadic nerve of the genuine equivariant operad associated to the simplicial $E \Sigma_G$-algebra $\mathsf O_G \wr \V^{\, \square}$.
      
      Since action groupoids $G \ltimes B$ are equivalent to disjoint unions of groups $\coprod_{[b] \in B/G} G_b$,
      the genuine equivariant operad $(\mathsf O_G \wr \V^{\, \square}, \otimes)$ is locally fibrant, and thus, by Theorem \ref{THMII}, $\UV_{\infty,G}^{\, \square}$ is in fact a $G$-symmetric monoidal $G$-$\infty$-category.
\end{definition}

We elaborate on this construction for a particular example.

\begin{example}
      Let $(\V, \otimes) = (\Top, \amalg)$ denote the category of compactly-generated spaces (with compactly-generated mapping spaces).
      Then for $G$-sets $U$, functors $G \ltimes U \xrightarrow{X} \Top$ are equivalent to maps of $G$-spaces $X \to U$,
      and under this presentation, the $E \Sigma_G$-algebra structure on $\mathsf O_G \wr \Top^\amalg$ takes the form
      \[
            \left( \left( \substack{X_U \\ \downarrow \\ U} \right)_{U \in A/G} \right)^{\otimes (A \to R)}
            =
            \left( \substack{\coprod X_U \\ \downarrow \\ R} \right).
      \]
      
      Let $\P$ denote the associated genuine equivariant operad, with colors and mapping spaces
      \[
            \mathfrak C_U = \Top^G \downarrow U,
            \qquad\qquad
            \P \left( \substack{A \\ \downarrow \\ R}, \left( \left(\substack{X_U \\ \downarrow \\ U} \right); \substack{Y \\ \downarrow \\ R} \right) \right)
            =
            \Map_{\Top^G \downarrow R}\left( \substack{\coprod X_U \\ \downarrow \\ R}, \substack{Y \\ \downarrow \\ R} \right).
      \]
      We see that the genuine category of operators $(\mathsf O_G \wr \Top^\amalg)^{\otimes} = \P^{\otimes}$ has objects of the form
      \[
            \left( \substack{A \\ \downarrow \\ R},\ (G \ltimes U \xrightarrow{X_U} \Top) \right),
            \qquad \mbox{or equivalently}
            \qquad
            \left( \substack{X \\ \downarrow \\ A \\ \downarrow \\ R} \right),
      \]
      and mapping spaces
      \begin{equation}
            \label{OGWRV_MAP_EQ}
            \Map \left( \substack{X \\ \downarrow \\ A \\ \downarrow \\ R},\ \substack{Y \\ \downarrow \\ B \\ \downarrow \\ S} \right) =
            \coprod_{(q, \bar f)} \prod_{V \in B/G} \Map_{\Top^G \downarrow V} \left( \substack{ X_V \\ \downarrow \\ V},\ \substack{ Y_V \\ \downarrow \\ V} \right),
      \end{equation}
      where $X_V$ is the fiber of $q^{\**}X$ over $\bar f^{-1}(V)$,
      and $Y_V$ is the fiber of $Y$ over $V$.
      Unpacking further, we see that the vertices of \eqref{OGWRV_MAP_EQ} are given by triples $(q, \bar f, \bar F)$, such that the following diagram commutes,
      where the two left-most and the top-middle squares are pullbacks.
      \[
            \begin{tikzcd}
                  X \arrow[d]
                  &
                  q^{\**} X \arrow[d] \arrow[l] \arrow[r, hookleftarrow] \arrow[dl, phantom, "\llcorner"'{very near start}]
                  &
                  X_f \arrow[d] \arrow[r, "\bar F"] \arrow[dl, phantom, "\llcorner"'{very near start}]
                  & 
                  Y \arrow[d]
                  \\
                  A \arrow[d]
                  &
                  q^{\**}A \arrow[d] \arrow[l] \arrow[r, hookleftarrow] \arrow[dl, phantom, "\llcorner"'{very near start}]
                  &
                  A_f \arrow[d] \arrow[r, "\bar f"]
                  &
                  B \arrow[d]
                  \\
                  Rq
                  &
                  S \arrow[l, "q"']
                  &
                  S \arrow[l, equal]
                  &
                  S \arrow[l, equal]
            \end{tikzcd}
      \]
      
      This construction recovers the $G$-symmetric monoidal $G$-$\infty$-category of $G$-spaces under disjoint union,
      as found in \cite{Hor}.
\end{example}

\begin{remark}
      Let $\V$ be any globally fibrant symmetric monoidal simplicial category.
      Following Corollary \ref{UNDERLY_COR},
      we note that:
      \begin{itemize}
      \item the underlying $\mathsf O_G$-$\infty$-category of $\UV_{\infty,G}^{\, \square}$
            is $N((\mathsf O_G \wr \V)^{op, \underline{op}})$, without it's monoidal structure,
            and in particular the fiber over $(G/G = G/G)$ is simply $N(\V^G)$, the coherent nerve of the category of strict $G$-objects in $\V$.
      \item The underlying symmetric monoidal $\infty$-category of $\UV_{\infty,G}^{\, \square}$
            is $N^\otimes((\mathsf O_G \wr V)_{G/G}, \otimes) = N^{\otimes}(\V^G, \square)$.
      \end{itemize}
\end{remark}

\begin{example}
      For $(V, \square) = (\Fin_{\**}, \amalg)$, the underlying $\mathsf O_G$-$\infty$-category of $\underline{\Fin}_{\**,\infty,G}^{\amalg}$ is precisely $\underline{\Fin}^G_{\**}$.
\end{example}

\begin{remark}
      There is a similarly named construction in parametrized higher category theory,
      the \textit{$G$-$\infty$-category of $G$-objects} from \cite[Defn. 7.4]{BDGNS}.
      Given any $\infty$-category $\mathcal D$, BDGNS define an $\mathsf O_G$-$\infty$-category $\underline{\mathcal D}_G \to \mathsf O_G^{op}$
      whose fiber over $G/G$ is equivalent to the functor $\infty$-category $\Fun(\mathsf O_G^{op}, \mathcal D)$.
      
      We warn that $\UV_{\infty,G}^{\, \square}$ is distinct from this notion applied to the infinity category $\mathcal D = N(\V)$,
      even after forgetting the monoidal structure.
      Specifically, consider the fibers over $G/G$,
      $N(\V^G)$ and $\Fun(\mathsf O_G^{op}, N(\V))$. 
      The objects in these categories differ in two important ways:
      \begin{itemize}
      \item[(i)] Objects in $N(\V^G)$ are simply objects with $G$-action, while those in $\Fun(\mathsf O_G^{op}, N(\V))$ are \textit{genuine $G$-objects}; and
      \item[(ii)] Objects in $N(\V^G)$ have a \textit{strict} $G$-action, while those in $\Fun(\mathsf O_G^{op}, N(\V))$ have a \textit{homotopy coherent} $\mathsf O_G^{op}$-action.
      \end{itemize}
\end{remark}

\subsection{Algebras over operads}

An algebra in a closed symmetric monoidal simplicial category $(\V, \square)$ over a simplicial operad $\O \in \sOp$ can be recovered as
a functor of simplicial operads $\O \to (\V,\square)$, where we identify $(\V,\square)$ with its image under the inclusion $\mathsf{sSymMon} \into \sOp$ that we extend in Proposition \ref{PERMG_OPG_PROP}.

Similarly, an algebra in the simplicial category of $G$-objects $\V^G$ over a simplicial $G$-operad $\O \in \sOp^G$
can be recovered as a functor of simplicial $G$-operads $\O \to (\V_G, \square)$,
where $\V_G$ is the $G$-enriched variation on $(\V, \square)$:
objects are $G$-objects in $\V$, with mapping $G$-spaces of \textit{all} arrows, with $G$ acting via conjugation.

In this short subsection, we prove Theorem \ref{THMIV}, which translates algebras over operads from the equivariant and simplicially-enriched setting to the $G$-$\infty$-categorical one.

We first define the categories in question.
\begin{definition}
      \label{ALG_DEF}
      Given an equivariant simplicial operad $\O \in \sOp^G$ and a symmetric monoidal simplicial category $\V$,
      define the simplicial category of \textit{$\O$-algebras in $\V^G$}, denoted $\Alg_{\O}(\V^G)$,
      to be the simplicial category of functors $\Fun_{\sOp^G}(\O, (\V_G,\square))$,
      with objects maps $F: \O \to (\V_G,\square)$ in $\sOp^G$,
      and mapping spaces
      \[
            \Nat(F,G) \subseteq \prod_{x \in \mathfrak C(\O)}\Map_{\V}(F(x), G(x))
      \]
      the subcomplex generated by the vertices $(\Phi_x)$ which form operadic natural transformations.
      i.e. for all $\phi \in \O(x_1,\dots,x_n;x_0)$, the diagram below commutes.
      \[
            \begin{tikzcd}
                  \mathop{\mathlarger{\mathlarger{\square}}}\limits_{n} F(x_i) \arrow[r, "\Phi"] \arrow[d, "{F(\phi)}"']
                  &
                  \mathop{\mathlarger{\mathlarger{\square}}}\limits_{n} G(x_i) \arrow[d, "{G(\phi)}"]
                  \\
                  F(x_0) \arrow[r, "\Phi_{x_0}"]
                  &
                  G(x_0)
            \end{tikzcd}
      \]
      
      For $\P \in \sOp_G$ and $E\Sigma_G$-algebra $\UV$, we analogously define simplicial categories
      \[
            \Alg_{\P}(\UV) = \Fun_{\sOp_G}(\P, (\UV,\otimes)),
            \qquad
            \Alg_{\P^{\otimes}}(\UV^{\otimes}) \subseteq \Fun_{\sCat \downarrow \underline{\Fin}^G_{\**}}(\P^{\otimes}, \UV^{\otimes}).
      \]
      We note that in the case where $\V$ or $\UV$ is locally fibrant, so are these simplicial categories.

      Finally, essentially by construction, we have simplicially-enriched comparison maps
      \[
            \Alg_{\O}(\V^G) \longto
            \Alg_{i_{\**}\O}(i_{\**}\V_G),
            \qquad
            \Alg_{\P}(\UV) \longto
            \Alg_{\P^{\otimes}}(\UV^{\otimes}).
      \]
\end{definition}
\begin{definition}
      \label{INFTYALG_DEF}
      For $\O^{\otimes}$ a $\mathsf O_G$-$\infty$-operad and $\mathcal C$ a $G$-symmetric monoidal $G$-$\infty$-category,
      define the $\infty$-category of \textit{$\O^{\otimes}$-algebras in $\mathcal C$}, denoted $\Alg_{\O^\otimes}(\mathcal C)$,
      to be the full subcomplex of $\Map_{\sSet \downarrow \underline{\Fin}^G_{\**}}(\O^{\otimes}, \mathcal C)$
      spanned by the maps of $\mathsf O_G$-$\infty$-operads.
\end{definition}

The following observation is the key step in the proof of Theorem \ref{THMIV}.

\begin{lemma}
      \label{ICG_LEM}
      There exists a natural transformation
      \[
            \begin{tikzcd}
                  \mathsf{sSymMon} \arrow[r, hookrightarrow, "{(-)_G}"] \arrow[d, hookrightarrow, "{\mathsf O_G \wr (-)}"']
                  &
                  \mathsf{sSymMon}^G \arrow[r, hookrightarrow]
                  &
                  \sOp^G \arrow[d, hookrightarrow, "i_{\**}"] \arrow[dll, Rightarrow, shorten <>=30pt]
                  \\
                  \mathsf{sSymMon}_G \arrow[rr, hookrightarrow] 
                  &&
                  \sOp_G
            \end{tikzcd}
      \]
\end{lemma}
\begin{proof}
      This follows from unpacking definitions. Let $(\V, \square)$ be a symmetric monoidal simplicial category.
      The genuine equivariant operad $i_{\**}(\V_G, \square)$ has object coefficient system the constant system at $\mathrm{Ob}(\V^G)$,
      and we define our natural transformation on objects by sending a pair $(U,X)$ to the diagram
      \[
            \Delta_U X\colon G \ltimes U \to \V,
            \qquad
            u \mapsto X,
            \quad
            (u \to g.u) \mapsto (X \xrightarrow{g} X).
      \]
      By \cite[Prop. 5.2]{BH_gsym}, there is a natural isomorphism of mapping spaces
      \[
            \begin{tikzcd}[row sep = small]
                  \displaystyle{
                    \Map_{i_{\**}(\V_G, \square)}\left(
                          \substack{ A \\ \downarrow \\ R},\ \big( (X_U); Y)
                    \right)
                    } \arrow[d, equal] \arrow[r]
                    &
                    \Map_{\mathsf O_G \wr \V}\left( \substack{A \\ \downarrow \\ R},\ \Big( \big(\Delta_U X_U \big); \Delta_R Y \Big) \right)
                    \arrow[d, equal]
                    \\
                    \displaystyle{                    
                      \left(\prod_{r \in R} \Map_{\V}\left( \mathop{\mathlarger{\mathlarger{\square}}}\limits_{C(a) = r} X_U, Y \right)\right)^G
                    }
                    \arrow[r, "\simeq"]
                    &
                    \Map_{\V^{G \ltimes R}}\left( C^{\otimes}_{\**}\big(\Delta_A X\big); \Delta_R Y \right).
              \end{tikzcd}
      \]
      where $X_U, Y \in \V^G$ for each $U \in A/G$.
      The result follows.
\end{proof}

\begin{proof}
      [Proof of Theorem \ref{THMIV}]
      First, given an algebra $\O \to (\V_G,\square)$, we have an associated composite of locally finite genuine equivariant operads
      \[
            i_{\**}\O \longto i_{\**}(\V_G, \square) \longto (\mathsf O_G \wr \V^{\, \square}, \otimes).
      \]
      By functoriality, this induces a map of $\mathsf O_G$-$\infty$-operads
      $N^\otimes(\O) \to \UV_{\infty,G}^{\, \square}$.
      
      Second, using Definition \ref{ALG_DEF} and Lemma \ref{ICG_LEM}, we have simplicial functors
      \[
            \Alg_{\O}(\V^G) \longto
            \Alg_{i_{\**}\O}(i_{\**}\V_G) \longto
            \Alg_{i_{\**}\O}(\mathsf O_G \wr \V^{\, \square}) \longto
            \Alg_{(i_{\**}\O)^{\otimes}}\left((\mathsf O_G \wr \V)^{\otimes}\right).
      \]
      
      Third, for any simplicial categories $\C$ and $\mathcal D$, we have a canonical map of simplicial sets
      \[
            N(\Fun(\C, \mathcal D)) \longto \Fun(N\mathcal C, N \mathcal D)
      \]
      produced over two adjoints via the composite
      \[
            \tau \left(N\Fun(\C,\mathcal D) \times N\mathcal C \right) \to
            \tau N\Fun(\C,\mathcal D) \times \tau N\mathcal C \xrightarrow{\epsilon}
            \Fun(\C, \mathcal D) \times \mathcal C \xrightarrow{ev}
            \mathcal D,
      \]
      where $\tau: \sSet \to \sCat$ is the left adjoint of the homotopy coherent nerve $N$.
      
      Combining these with Definition \ref{INFTYALG_DEF}, we produce a functor of $\infty$-categories as desired.
      \[
            N\Alg_\O(\V^G) \longto \Alg_{N^\otimes\O}(\UV_{\infty,G}^{\, \square})
      \]
\end{proof}


\newcommand{\etalchar}[1]{$^{#1}$}
\providecommand{\bysame}{\leavevmode\hbox to3em{\hrulefill}\thinspace}
\providecommand{\MR}{\relax\ifhmode\unskip\space\fi MR }
\providecommand{\MRhref}[2]{%
  \href{http://www.ams.org/mathscinet-getitem?mr=#1}{#2}
}
\providecommand{\doi}[1]{%
  doi:\href{https://dx.doi.org/#1}{#1}}
\providecommand{\arxiv}[1]{%
  arXiv:\href{https://arxiv.org/abs/#1}{#1}}
\providecommand{\href}[2]{#2}

\end{document}
